\documentclass[11pt,reqno]{amsproc}

\usepackage{amsmath,amsfonts,amssymb,amsthm}
\usepackage[abbrev,lite,nobysame]{amsrefs}
\usepackage{mathrsfs}
\usepackage{graphics,graphicx}
\usepackage[usenames,dvipsnames]{color}
\usepackage{times}
\usepackage{bbm}
\usepackage[margin=1in]{geometry}
\usepackage[colorlinks=true, pdfstartview=FitV, linkcolor=blue, citecolor=blue, urlcolor=blue]{hyperref}

\makeatletter

\renewcommand\subsection{\@startsection{subsection}{2}%
\normalparindent{.5\linespacing\@plus.7\linespacing}{-.5em}
{\normalfont\bfseries}}

\renewcommand\subsubsection{\@startsection{subsubsection}{3}%
\normalparindent{.5\linespacing\@plus.7\linespacing}{-.5em}
{\normalfont\bfseries}}

\def\@tocline#1#2#3#4#5#6#7{\relax
  \ifnum #1>\c@tocdepth 
  \else
    \par \addpenalty\@secpenalty\addvspace{#2}%
    \begingroup \hyphenpenalty\@M
    \@ifempty{#4}{%
      \@tempdima\csname r@tocindent\number#1\endcsname\relax
    }{%
      \@tempdima#4\relax
    }%
    \parindent\z@ \leftskip#3\relax \advance\leftskip\@tempdima\relax
    \rightskip\@pnumwidth plus4em \parfillskip-\@pnumwidth
    #5\leavevmode\hskip-\@tempdima
      \ifcase #1
       \or\or \hskip 1em \or \hskip 2em \else \hskip 3em \fi%
      #6\nobreak\relax
    \dotfill\hbox to\@pnumwidth{\@tocpagenum{#7}}\par
    \nobreak
    \endgroup
  \fi}
\makeatother

\newtheorem{theorem}{Theorem}
\newtheorem{proposition}{Proposition}[section]
\newtheorem{lemma}[proposition]{Lemma}

\theoremstyle{definition}
\newtheorem{definition}[proposition]{Definition}

\newtheorem{remark}[proposition]{Remark}

\numberwithin{equation}{section}

\newcommand\eps{\varepsilon}
\newcommand\e{{\rm e}}
\newcommand\dd{{\rm d}}
\newcommand\ddt{{\frac{\dd}{\dd t}}}

\newcommand\spt{{\rm spt}}

\def\l {\langle}
\def\r {\rangle}

\newcommand{\norm}[1]{\left\lVert #1 \right\rVert}

\newcommand\TT {{\mathbb T}}
\newcommand\RR {{\mathbb R}}
\def\SS {{\mathbb S}}

\newcommand\cL{{\mathcal L}}

\newcommand\cR{{\mathcal R}}
\newcommand\cS{{\mathcal S}}

\usepackage{subcaption}
\usepackage{enumerate}
\usepackage{enumitem}
\usepackage{mathtools}

\newcommand{\haus}{\mathscr{H}}
\newcommand{\sfp}{\mathsf{p}}
\newcommand{\sfc}{\mathsf{c}}
\mathtoolsset{showonlyrefs=true}

\begin{document}

\title[Enhanced dissipation for two-dimensional Hamiltonian flows]{Enhanced dissipation for two-dimensional Hamiltonian flows}

\author[E. Bru\`e]{Elia Bru\`e}
\address{School of Mathematics, Institute for Advanced Study, 1 Einstein Dr., Princeton NJ 05840, USA}
\email{elia.brue@math.ias.edu}

\author[M. Coti Zelati]{Michele Coti Zelati}
\address{Department of Mathematics, Imperial College London, London, SW7 2AZ, UK}
\email{m.coti-zelati@imperial.ac.uk}

\author[E. Marconi]{Elio Marconi}
\address{Dipartimento di Matematica `Tullio Levi Civita', Universit\`a di Padova, via Trieste 63, 35121 Padova (PD), Italy}
\email{elio.marconi@unipd.it}



\begin{abstract}
Let $H\in C^1\cap W^{2,p}$ be an autonomous, non-constant Hamiltonian on a compact $2$-dimensional manifold, generating an
incompressible velocity field $b=\nabla^\perp H$. We give sharp upper bounds on the enhanced dissipation rate
of $b$ in terms of the properties of the period $T(h)$ of the close orbits $\{H=h\}$. Specifically, if $0<\nu\ll 1$ is the
diffusion coefficient, the enhanced dissipation rate can be at most $O(\nu^{1/3})$ in general, the bound
improves when $H$ has isolated, non-degenerate elliptic point. Our result provides the better bound $O(\nu^{1/2})$
for the standard cellular flow given by $H_\mathsf{c}(x)=\sin x_1 \sin x_2$, for which we can also prove a new upper bound
on its mixing mixing rate and a lower bound on its enhanced dissipation rate. 
The proofs are based on the use of action-angle coordinates and on the existence of a good invariant domain for
the regular Lagrangian flow generated by $b$.
\end{abstract}

\maketitle

\setcounter{tocdepth}{1}
\tableofcontents

\section{Advection-diffusion equations}
Let $(M,g)$ be a compact $2$-dimensional manifold, possibly with boundary, and consider 
an autonomous, non-constant Hamiltonian $H$ which generates a velocity field $b:=\nabla^\perp H$,  tangent to the boundary whenever $\partial M \neq \emptyset$.
We are interested in the long-time dynamics of the scalar function $\rho^\nu:[0,\infty)\times M\to\RR$ subject to the advection-diffusion equation
\begin{equation}\label{e:adv-diff}
	\begin{cases}
		\partial_t \rho^\nu + b\cdot \nabla \rho^\nu = \nu \Delta \rho^\nu, \\
		\rho^\nu(0,\cdot) = \rho_0.
	\end{cases}\tag{A-D}
\end{equation}
Here, $\rho_0:M\to\RR$ is an assigned mean-free initial datum and $\nu>0$ is the diffusivity parameter.  When $\partial M \neq \emptyset$,
we prescribe homogeneous Neumann condition $\partial_n \rho^\nu =0$, where $n$ is the outward normal derivative to the boundary.

In the non-diffusive case, i.e. when $\nu=0$, \eqref{e:adv-diff} reduces to the standard transport equation
\begin{equation}\label{e:transport}
	\begin{cases}
		\partial_t \rho + b\cdot \nabla \rho = 0, \\
		\rho(0,\cdot) = \rho_0.
	\end{cases}\tag{T}
\end{equation}
The goal of this article is to study the mixing and diffusive properties of  \eqref{e:adv-diff} and \eqref{e:transport} in terms of sharp decay rates for $\rho^\nu$ and $\rho$,
under general assumptions on the Hamiltonian $H$.

\subsection{Mixing and enhanced dissipation}
Enhanced dissipation typically refers to the accelerated decay of solutions to \eqref{e:adv-diff} due to the interaction of transport and diffusion.
We are interested in putting on sound mathematical grounds the following statement from \cite{RhinesYoung83}:

\begin{quote}
The homogenization of a passive tracer  in a flow with closed mean streamlines occurs
in two stages: first, a rapid phase dominated by shear-augmented diffusion over a
time $\sim \nu^{-1/3}$, in which initial values of the tracer are replaced by
their (generalized) average about a streamline; second, a slow phase requiring the full
diffusion time $\sim \nu^{-1}$.
\end{quote}
The above statement can be interpreted in terms of the behavior of the $L^2$ norm of the solution $\rho^\nu$ to  \eqref{e:adv-diff}. 
In view of the energy balance
\begin{align}
\ddt \|\rho^\nu\|^2 + 2\nu \|\nabla\rho^\nu\|^2=0, \qquad \forall t\geq 0,
\end{align}
all mean-free solutions decay exponentially to zero as $\e^{-c_\sfp\nu t}$, where $c_\sfp>0$ is related to the Poincar\'e constant. Hence, the
natural diffusive time-scale $O(\nu^{-1})$ appears trivially, and no role is played by the velocity field $b$. Now, the above-mentioned \emph{slow phase}
refers to such diffusive behavior for the average of  $\rho^\nu$ on the streamlines of the Hamiltonian $H$: indeed, if $\rho^\nu$ were constant on the
streamlines, it would then follow that $b\cdot \nabla \rho^\nu=0$, implying that the diffusive behavior is the only possible one. On the contrary, 
the rest of the solution is conjectured to undergo the \emph{rapid phase}, in which decay happens on a much faster time-scale $O(\nu^{-1/3})$.
These considerations are  at the heart of the concept of \emph{enhanced dissipation}, formalized in the definition below.
 
\begin{definition}\label{def:diff-enhancing}
Let $\nu_0\in(0,1)$ and $\lambda:(0,\nu_0)\to (0,1)$ be a continuous increasing function such that
\begin{align}
\lim_{\nu\to 0}\frac{\nu}{\lambda(\nu)}=0.
\end{align} 
The velocity field $b$ is \emph{dissipation enhancing} at rate $\lambda(\nu)$ if
there exists $A\geq 1$ only depending on $b$ such that if $\nu\in(0, \nu_0)$ then for every $\rho_0\in L^2$ with zero streamlines-average we have the enhanced dissipation estimate
\begin{align}\label{eq:decay}
\|\rho^\nu(t)\|_{L^2}\leq A \e^{- \lambda(\nu)t}\|\rho_0\|_{L^2}, 
\end{align}
for every $t\geq 0$.
 \end{definition}
While the above concept has been more or less informally studied in the physics literature since the late nineteenth century, 
it has received much attention by the mathematical community only recently, starting with the seminal article \cite{CKRZ08}. In this work, enhanced dissipation 
has been proven to be equivalent to the non-existence of non-trivial $H^1$-eigenfunctions of the transport operator $b\cdot \nabla$: in particular, 
functions that are constant on streamlines are eigenfunctions and hence have to be excluded when studying enhanced dissipation in the two-dimensional, autonomous setting.
From a quantitive point of view, the picture is now quite clear in the context of shear flows \cites{BCZ17,BW13,CCZW21,Wei21,ABN22,CZG21,CZD21} and radial flows \cites{CZD20,CZD21}. For more general velocity fields, there are only some
results linking mixing rates and enhanced dissipation time-scales \cites{CZDE20,FI19}, and others that study the interplay between regularity and dissipation \cites{BN21}. We also mention 
the interesting work \cite{Vukadinovic21}, which deals with enhanced dissipation for an averaged equation stemming from general hamiltonians. However, 
a precise quantitative picture is still missing.

The goal of this article is to analyze enhanced dissipation in the case of velocity fields originating from general (regular) Hamiltonians. 
According to Definition \ref{def:diff-enhancing}, the case $\lambda(\nu)=c_\sfp\nu^{1/3}$ is precisely the one described in \cite{RhinesYoung83}. 
One of our main results is that the exponent $1/3$ is the best possible in the autonomous setting.

\begin{theorem}\label{thm:main}
 Let $p\geq 2$ and $H\in C^1\cap W^{2,p}(M)$ be such that $b:=\nabla^\perp H$ is dissipation enhancing with rate $\lambda(\nu)$. Then
	\begin{equation}\label{eq:upperrnu}
		\lambda(\nu) \le C_0 \nu^{1/3}.
	\end{equation}
for some positive constant $C_0=C_0(H)$.
\end{theorem}
In the language of \cite{RhinesYoung83}, we prove that the rapid phase cannot happen before a time-scale $O(\nu^{-1/3})$. The exponent $1/3$
is not always the correct one, at least in the sense of Definition \ref{def:diff-enhancing}. It is achieved, for instance, by the shear flow $b=(x_2,0)$
on $\SS^1\times [0,1]$, namely the Couette flow \cites{Kelvin87}. However, it is well-known that the presence of critical points can slow down the dissipation \cite{BCZ17}.
This particular exponent is due to
dimensionality, regularity and the autonomous nature of our problem. Indeed:
\begin{itemize}
\item  contact Anosov flows on smooth odd dimensional connected compact Riemannian manifolds
 have an enhanced dissipation time-scale $O(|\ln\nu|^2)$, see \cite{CZDE20};

\item there exist H\"older continuous shear flows on the two-dimensional torus with an enhanced dissipation
time-scale $O(\nu^{-\gamma})$, for any $\gamma\in (0,1)$, see \cite{CCZW21,Wei21}

\item non-autonomous velocity fields generated as solutions to the stochastic Navier-Stokes equations  on the two-dimensional torus  have an enhanced dissipation
time-scale $O(|\ln\nu|)$, see \cite{BBPS21}.
\end{itemize}
The proof of Theorem \ref{thm:main}
is based on the following key observation: any Hamiltonian $H\in C^1\cap W^{2,p}(M)$ has an invariant domain $\Omega$ where the gradient of the flow map grows at most linearly in time. It follows that
\begin{equation}\label{z6}
	\| \rho^\nu (t) - \rho(t)\|_{L^2}^2 
	\lesssim 
	\nu t^3 \qquad \forall t\ge 1\, ,
\end{equation}
provided $\rho_0$ is concentrated in $\Omega$.
In view of the inviscid conservation $\| \rho(t)\|_{L^2} = \| \rho_0 \|_{L^2}$, estimate \eqref{z6} provides the desired upper bound on $\lambda(\nu)$ by choosing $t\sim \nu^{-1/3}$.

As mentioned earlier, the enhanced dissipation properties of $b$ are closely linked to its mixing features, as defined below.

\begin{definition}\label{def:mixing rate}
  Let $\gamma: [0,+\infty) \to [0,+\infty)$ be a continuous and decreasing function vanishing at infinity. The velocity field $b$ is \textit{mixing} with rate $\gamma(t)$ if for every $\rho_0\in H^1$ with zero streamlines-average we have the following estimate
  \begin{equation}\label{eq:mixingdefo}
  	\| \rho(t) \|_{H^{-1}} \le \gamma(t) \| \rho_0 \|_{H^1}\, ,
  \end{equation}
  for every $t\geq 0$.
\end{definition}

Due to the conservation of the $L^2$ norm in the transport equation \eqref{e:transport}, the mixing estimate \eqref{eq:mixingdefo} implies that the $H^1$ norm of $\rho$
has to grow at least as $1/\gamma(t)$. This has been already observed in the case of shear flow with critical points \cite{BCZ17}, and we will provide another example
when dealing with the standard cellular flow below (see Section \ref{sub:cellflow}).

\subsection{Elliptic points}
\label{subsec:elliptic}

In the presence of elliptic points, mixing rates can be slower and in turn affect the enhanced dissipation time-scale. This is explicit 
in the case of shear flows $b=(v(x_2),0)$ on $\TT^2$, where the crucial role is played by the order of vanishing of derivatives of $v$
at critical points \cite{BCZ17}. The case of a velocity field generated by a Hamiltonian is in general much more complicated, as 
level sets of $H$ are not, as for shear flows, simply horizontal lines (see Figure \ref{fig:hamilton}).
\begin{figure}[h!]
  \centering
  \begin{subfigure}[b]{0.45\linewidth}
    \includegraphics[width=\linewidth]{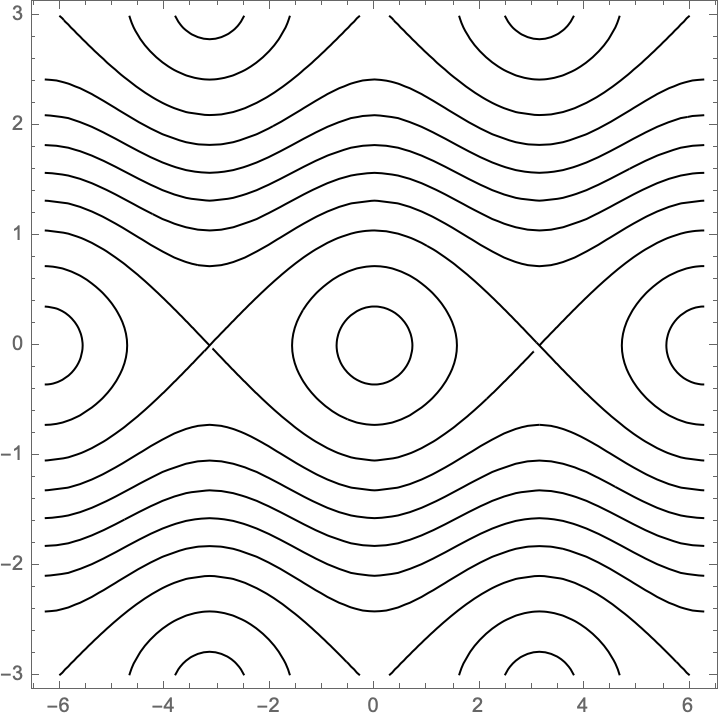}
     \caption{A general hamiltonian on a periodic domain.}
       \label{fig:hamilton}
  \end{subfigure}
\hfill
  \begin{subfigure}[b]{0.45\linewidth} 
    \includegraphics[width=\linewidth]{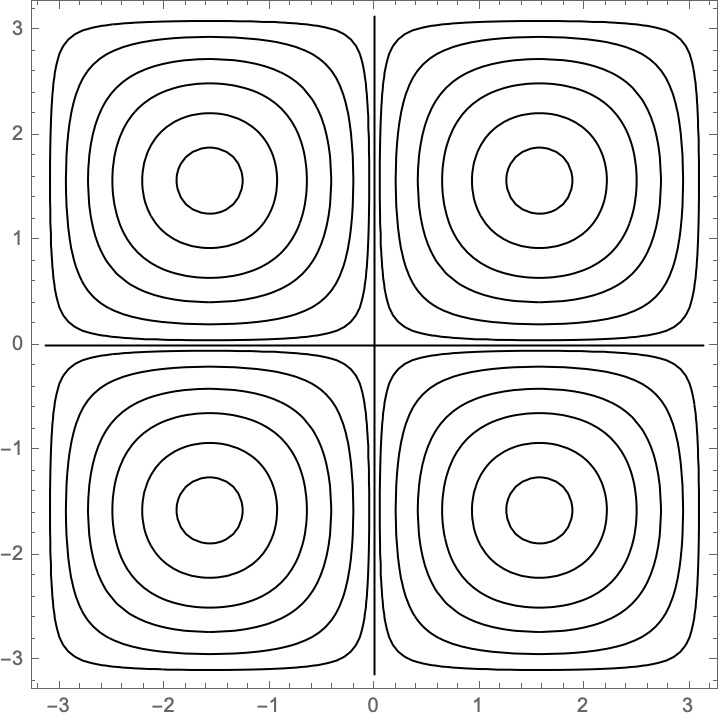}
    \caption{The standard cellular flow.}
     \label{fig:cellularstand}
  \end{subfigure}
\end{figure}

By compactness, $H$ always admits a minimum point in $p\in M$. To avoid degeneracies, we assume that $p\notin \partial M$ is an isolated elliptic point, i.e. $\nabla^2 H(p)$ is positive definite. To study the system around $p\in M$, we use local coordinates such that $p=(0,0)$.

Let $r_0>0$ be fixed (possibly small), and let $r\in (0,r_0)$.  Consider the unique closed orbit $\ell_r=\{x\in M: H(x)=H(re_1)\}\subset M$ 
containing the point $re_1=(r,0)$, and denote by $T(r)$ its period, which is a $C^1$ function since $r\to H(r e_1)$ is $C^1$ and $T$ is a $C^1$ function of $H$. 
By using that $H(re_1)\sim r^2$, $\|b\|_{L^\infty(\ell_r)}\lesssim r$, and \eqref{eq:period} we immediately get $T(0)>0$.

\begin{theorem}\label{thm:elliptic}
Let $H\in C^3(M)$ have an isolated, non-degenerate elliptic point $p\in M\setminus \partial M$. Assume that $T'(r)\sim r^{\beta}$ for some $\beta \ge 0$.
Assume that $b:= \nabla^{\perp} H$ is mixing (resp. dissipation enhancing) with rate $\gamma(t)$ (resp. $\lambda(\nu)$). Then
\begin{align}
	\gamma(t)\geq \frac{C_0}{\l t\r^{\frac2{\beta+1}}} \qquad \mbox{and} \qquad \lambda(\nu)\leq C_0\nu^{\frac{1+\beta}{3+\beta}}
\end{align}
for some positive constant $C_0=C_0(H)$.
\end{theorem}

These estimates are obtained by studying the dynamics around the elliptic point of $H$. This local analysis improves the available bound on the mixing rate for Hamiltonian flows when $\beta >1$ (see \cite{BM21}), as well as the bound on the dissipation enhancing rate as soon as $\beta >0$, compare with Theorem \ref{thm:main}.

The cellular flow $H_\mathsf{c}(x)=\sin x_1 \sin x_2$ satisfies the assumption of the theorem with $\beta=1$.

\begin{remark}
The heuristic built on the case of shear flows, which allows to deduce the dissipation enhancing rate $\lambda(\nu) \sim \nu^{\frac13}$ from the mixing rate $\gamma(t) \sim t^{-1}$ does not match with the previous result: this difference can be attributed to the different geometry of the level sets of $H$ in a shear flow and around an elliptic point. 
\end{remark}

\subsection{Cellular flows}\label{sub:cellflow}
Cellular flows (along with shear and radial flows) are perhaps the most studied two-dimensional flows,
especially from the point of view of fluid dynamics, homogenization and as random perturbations of
dynamical systems \cites{RhinesYoung83,Koralov04,Heinze03,NPR05,FP94,CS89}. Strictly related to the idea of dissipation enhancement, there has been a number of articles \cites{IZ22,FX22,IXZ21} 
dealing with suitable rescalings of cellular flows (which create small scales, and hence roughness in the velocity field) and proving that the dissipation time can be made arbitrarily small by taking the rescaling parameter
small. 

On the contrary, we here consider the standard cellular flow $H_{\sfc}(x) = \sin x_1 \, \sin x_2$ on $\TT^2$, as depicted in Figure  \ref{fig:cellularstand}, and 
prove a direct estimate on the mixing and enhanced dissipation rates.
\begin{theorem}\label{thm:cell}
Consider the standard cellular flow $b_{\sfc}= \nabla^{\perp}H_{\sfc}$ with  $H_{\sfc}(x) = \sin x_1 \, \sin x_2$ on $\TT^2$. Then for every $\eps>0$ the vector field $b_{\sfc}$ is mixing 
with rate
\begin{equation}\label{eq:mixingcell}
\frac{1}{\l t\r}\lesssim \gamma_{\sfc}(t) \lesssim_\eps \frac1{\l t\r^{\frac13-\eps}},
\end{equation}
and dissipation enhancing with  rate
\begin{equation}\label{eq:enhancedcell}
 \nu^{6/7+\eps}\lesssim_\eps\lambda_{\sfc}(\nu)\lesssim \nu^{1/2}.
\end{equation}
\end{theorem}
The lower bound on the mixing rate and the upper bound on the enhanced dissipation rate are a consequence of Theorem \ref{thm:elliptic}.
As we shall see, our approach is based on a careful study of the action-angle coordinates near the elliptic and hyperbolic points of $H_{\sfc}$, \emph{without}
making use of  the well-known estimates of the coefficient of effective diffusivity of $H_{\sfc}$ from homogenization theory \cites{FP94,Koralov04,CS89}. We stress that
the rates \eqref{eq:mixingcell} and  \eqref{eq:enhancedcell} are \emph{global} in nature, and do not distinguish where the initial datum is supported. 
In fact (see Section  \ref{sub:globalmix}), the upper bound in \eqref{eq:mixingcell} and the lower bound in \eqref{eq:enhancedcell} come from data that are supported near
the hyperbolic points of $H_{\sfc}$, as these are the hardest to control with our methods.  Near elliptic points, the solution is mixed at the faster rate
$\l t\r^{-1+\eps}$, where $\eps>0$ can be taken arbitrarily small, while away from both elliptic and hyperbolic points, the mixing rate improves to $\l t\r^{-1}$. 
In particular, the  lower bound in \eqref{eq:mixingcell} is sharp.

\section{Gradient estimates for transport equations}
In this section, we focus on the transport equation \eqref{e:transport}
and its associated flow $X:\mathbb{R}\times M\to M$ defined as the solution to the ODE 
\begin{align}\label{e:Lagr}
\begin{cases}
\partial_t X(t,x) = b(t,X(t,x)),\\
X(0,x) = x .
\end{cases}
\end{align}
Besides some regularity assumptions, which will be specified later, we assume that $H=0$ on $\partial M$ whenever $\partial M \neq \emptyset$.
 We will call a set $E\subset M$ \emph{invariant} (under the flow $X$) if $X(t,E)=E$ for every $t\geq 0$. The key step in the proof of Theorem \ref{thm:main} 
 is the existence of a good invariant set for $X$.

  \begin{proposition}\label{prop:keystep}
 Let $H\in C^1\cap W^{2,p}(M)$ for some $p\geq 1$.
  There exists an invariant  open set $\Omega 	\subset M$ such that for any $\rho_0 \in C^1(M)$ with 
	$\spt (\rho_0) \subset \Omega$, the corresponding solution $\rho$ of \eqref{e:transport} satisfies  
  	\begin{equation}\label{eq:reg}
  		\| \nabla \rho( t) \|_{L^p}
  		\le C(\Omega, H) (1+t)\|\nabla \rho_0\|_{L^\infty},
  	\end{equation}
for all $t\geq 0$.
  \end{proposition}
Despite the possible presence of hyperbolic points in $H$, the set $\Omega$ is one where only shearing is possible. As a consequence, the growth of $\nabla \rho$
is limited to be linear in time, excluding for instance exponential growth.

\subsection{Construction of a good invariant domain}
To build a suitable invariant set, the idea is to employ a variation of the classical Morse-Sard lemma \cite{Morse39,Sard42} proven in \cite{BKK13}. 

\begin{lemma}\label{lemma:invariant set}
 Let $H\in C^1\cap W^{2,p}(M)$ for some $p\geq 1$.
	There exist a constant $c_0>0$ and an interval $(h_0,h_1)\subset H(M)$ such that $|b(x)|\ge c_0$ for any $x$ in the invariant set $\Omega:= H^{-1}((h_0,h_1))$.
	Moreover,  $\Omega \cap \partial M = \emptyset$.
\end{lemma}

\begin{proof}
Define the set of critical values of $H$  by
\begin{equation}
	\cS:= \{ h \in \RR: \exists\, x \in M \mbox{ with } H(x)=h \mbox{ and }b(x)=0 \},
\end{equation}
 and by $\cR := H(M)\setminus \cS$ the set of regular values. 
Since $H\in C^1(M)$ and $M$ is compact, the set $\cS$ is closed and therefore $\cR$ is open in the range of $H$.
As shown in \cite{BKK13},  such $H$ enjoys the Sard property, namely $\cS$ has zero 1-dimensional Lebesgue measure on $\RR$. The latter can be reduced to a statement on a local chart, hence the fact that $M$ is not a subset of $\RR^2$ is irrelevant for the sake of applying \cite{BKK13}.
In particular, the sets $\cR$ and $\Omega_*:=H^{-1}(\cR)$ are nonempty open subsets of $H(M)$ and $M$, respectively.
It is straightforward to check that the map $c_S: \cR \to \RR$ defined by
\begin{equation}
	c_S(h)=  \min_{x\in H^{-1}(h)}|b(x)|
\end{equation}
is continuous on $\cR$. 
Being $H$ non-constant, there exists $c_0>0$ such that $c_S(h) >2c_0$ for some $h\in \cR$. 
Hence, we can find and interval $(h_0,h_1)\subset \cR$ with $h\in (h_0,h_1)$ such that $c_S>c_0$ on $(h_0,h_1)$.  We now set $\Omega:= H^{-1}((h_0,h_1))$. 
If $\partial M \cap \Omega = \emptyset$, we are done.  Otherwise, since $H$ vanishes on the boundary, $0\in (h_0,h_1)$. Replacing $(h_0,h_1)$ 
with smaller interval $(h_0',h_1')$ not containing 0, we can redefine   $\Omega= H^{-1}((h_0',h_1'))$. The fact that $\Omega$ is invariant follows
from its definition, 
and hence the proof is over.
\end{proof}

\begin{remark}
At this stage, the above lemma holds true for more general Lipschitz Hamiltonians whose gradient is a function of bounded variation, see \cite{BKK13}. 
\end{remark}

\subsection{Action-angle coordinates for $C^2$ Hamiltonians}\label{ss:action-angle}
Assume that $H\in C^2(M)$, and let $\Omega$ and $c_0>0$ as in Lemma \ref{lemma:invariant set}. 
    Let $ x_0 \in M$ be such that $H( x_0)=h_0$ and denote by $\Omega_0$ the connected component of $\Omega$ 
    such that $ x_0 \in\partial \Omega_0$. Given $h\in (h_0,h_1)$, we denote the period (relative to the flow map $X$ in \eqref{e:Lagr}) 
    of the closed orbit $\{H = h\}$ by $T(h)$, while $x=x(h):(h_0,h_1)\to M$ stands for the solution to the  ODE  
	\begin{equation}
		\begin{cases}
			x'(h)  = \frac{\nabla H}{|\nabla H|^2}(x(h)),\\
			x(h_0) =   x_0.
		\end{cases}
	\end{equation}
    Using, the flow map $X$ in \eqref{e:Lagr}, we define the coordinates $\Phi : \mathbb{S}^1 \times (h_0,h_1) \to \Omega$ by
	\begin{equation}
		\Phi(\theta, h) := X(\theta T(h), x(h)).
	\end{equation}
	Here $\mathbb{S}^1=[0,1)$. Notice that $H(x(h)) = h$ for any $h\in (h_0,h_1)$. The key properties of $\Phi$ are contained in the following proposition.

	\begin{lemma}\label{l:Phi}
		The map $\Phi : \mathbb{S}^1 \times (h_0,h_1) \to \Omega_0$ and its inverse are $C^1$ functions. Moreover, $\Phi^{-1}(x) = (\Psi(x), H(x))$ where $\Psi : \Omega_0 \to \mathbb{S}^1$ is a $C^1$ function satisfying 
		\begin{equation}\label{e:evolution_psi}
			\Psi(X(t ,x)) = \Psi(x) + \frac{t }{T(H(x))} ,
		\end{equation}
		for any $x\in \Omega$, $t\geq 0$.
	\end{lemma}

\begin{proof}
 We begin by proving that $T:(h_0,h_1) \to (0,\infty)$ is $C^1$.
		 It holds
		 \begin{align}
		 	T(h) = \int_{\{H=h\}{\cap \Omega_0}} \frac{1}{|b|} \dd \haus^1
		 	 & =
		 	\int_{\{h_0\le H\le h\}{\cap \Omega_0}} \nabla\cdot\left(\frac{\nabla H}{|\nabla H|^2}\right) \dd \haus^2 + T(h_0) \notag\\
		 	& =
		 	\int_{h_0}^h \int_{\{H=r\}{\cap \Omega_0}} \nabla\cdot\left(\frac{\nabla H}{|\nabla H|^2}\right) |\nabla H|^{-1} \dd\haus^1\, \dd r
		 	 + T(h_0),\label{eq:period}
		 \end{align}
	     where we used the divergence theorem and the coarea formula.
	     Using that $|\nabla H|\ge c_0$ in $\Omega$, it is now immediate to see that $T\in C^1((h_0,h_1))$.

The fact that $b\in C^1(M)$ implies that $\Phi$ is $C^1$-regular. More precisely, pointwise in $(\theta,h)\in \mathbb{S}^1 \times (h_0,h_1)$  we have
	     \begin{align}
	     	\partial_h \Phi(\theta,h)
	     	= T'(h) \nabla^\perp H (X(\theta T(h), x(h)))+ D_x X(\theta T(h),x(h))  \frac{\nabla H}{|\nabla H|^2}(x(h)) 
	     \end{align}
and
	     \begin{align}\label{eq:thetader}
	     	\partial_\theta \Phi(\theta,h)= T(h) \nabla^\perp H(\Phi(\theta,h)),
	     \end{align}
so that standard regularity estimates for flow maps of $C^1$ velocity fields imply
	     \begin{align}
	     	|\partial_h \Phi(\theta,h)| 
	     	\le
	     	|T'(h)| |b|(\Phi(\theta,h)) + \frac{\e^{T(h)\|H\|_{C^2}}}{|b|(x(h))}
	     \end{align}
and
	     \begin{align}
	     	|\partial_\theta \Phi(\theta,h)|
	     	&\le T(h) |b|(\Phi(\theta,h)).
	     \end{align}
	     It is simple to see that $\Phi$ is injective and surjective. To prove that the inverse is $C^1$ we show that $D\Phi$ is invertible at any point. Now,
	     from the identity
	     \begin{equation}
	     	H(\Phi(\theta,h)) = H(x(h)) = h,
	     \end{equation}	    
and \eqref{eq:thetader} we deduce that
	     \begin{equation}\label{z1}
	     	\begin{cases}
	     		\partial_h \Phi \cdot \nabla H(\Phi) = 1 \, ,
	     		\\
	     		\partial_\theta \Phi =  T(h)\nabla^\perp H (\Phi),
	     	\end{cases}
	     \end{equation}
	     for any $\theta \in \mathbb{S}^1$ and $h\in (h_0,h_1)$. Since $(\nabla H, \nabla^\perp H)$ is an orthogonal and non-degenerate frame, the $C^1$ property of $\Phi^{-1}$ follows. Specifically,
	     from  \eqref{z1} we find
	     \begin{equation}\label{e:Jacobian_angle-action}
	     	|\det D\Phi(\theta,h)| = T(h) ,
	     \end{equation}
	      and therefore 
	     \begin{align}
	     	|D \Phi^{-1}|(\Phi(\theta,h))
	     	 \le \frac{|D \Phi|(\theta,h)}{|\det D\Phi(\theta,h)|} \le 
	     	 \left(\frac{|T'(h)|}{T(h)} + 1 \right) |b|(\Phi(\theta,h)) + \frac{1}{T(h)} \frac{\e^{T(h)\|H\|_{C^2}}}{|b|(x(h))}\, .
	     \end{align}   
         In order to prove \eqref{e:evolution_psi}, observe that $x= \Phi(\Psi(x),H(x)) = X(\Psi(x) T(H(x)),x(H(x)))$, and therefore
	     \begin{align}
         X(t,x) = X(\Psi(x) T(H(x)) + t, x(H(x))), \qquad \forall t\geq0.
	     \end{align}  
         On the other hand
	     \begin{align}
         X(t,x) = \Phi(\Psi(X(t,x)),H(X(t,x))) =  \Phi(\Psi(X(t,x)),H(x)) = X(\Psi(X(t,x))T(H(x)),x(H(x))).
	     \end{align}  
Comparing the two expressions above and recalling that $s\to X(sT(h),h)$ is 1-periodic and injective on $[0,1)$   we obtain \eqref{e:evolution_psi}
and complete the proof.
	\end{proof}
With the above lemma at hand, and in particular the explicit formula \eqref{e:evolution_psi}, the proof of Proposition \ref{prop:keystep} follows immediately.

	\begin{proof}[Proof of Proposition \ref{prop:keystep} when $H\in C^2$]
	We denote by $\dd_M(\cdot,\cdot)$ the Riemannian distance on $M$.
	Invoking  \eqref{e:evolution_psi},
	for any $x,y\in \Omega_0$ and $t\geq 0$ it holds 
	\begin{align}
		\dd_M(X(t,x), X(t,y)) 	
		& = \dd_M \left( \left(\Phi \left(\Psi(X(t,x)) \right),H(X(t,x))\right), \Phi\left(\left(\Psi(X(t,y))\right),H(X(t,y))\right)\right) \notag\\
		& = \dd_M \left( \left(\Phi \left(\Psi(x) + \frac{t}{T(H(x))}\right),H(x)\right), \Phi\left(\left(\Psi(y) + \frac{t}{T(H(y))}\right),H(y)\right)\right) \notag\\
		& \lesssim |\Psi(x) - \Psi(y)| + t \frac{|T(H(x)) - T(H(y))|}{T(H(x))T(H(y))} + |H(x) - H(y)|\notag\\
		& \lesssim (1+t)\dd(x,y),
	\end{align}
	where we used that the period $T$ is bounded below away from zero on $(h_0,h_1)$ thanks to \eqref{eq:period}. 
	In particular, if $\rho_0\in C^1$ is supported in $\Omega$, then $\rho(t,x) = \rho_0(X(-t,x))$ and
	\begin{equation}
		\| \nabla \rho(\cdot,t) \|_{L^p} \lesssim \| \nabla \rho_0 \|_{L^\infty} (1+t) ,
	\end{equation}
	for any $p \in[1,\infty]$, thereby concluding the  proof. 
\end{proof}

\subsection{Less regular Hamiltonians}
When $H\in C^1\cap W^{2,p}(M)$ we cannot appeal to action-angle variables to study the flow map $X$ on good invariant domains. To be precise, we cannot even appeal to classical notions of flow since $b\in C^0\cap W^{1,p}(M)$ is not regular enough.
	
	In this setting, we understand $X$ as the unique \textit{regular Lagrangian flow} (RLF in short) associated to $b$ in the sense of \cite{DiPernaLions89,Ambrosio04}. 
	The latter is by definition a measurable map $X: \RR \times M \to M$ satisfying the following properties:
	\begin{itemize}
		\item[(i)] $X(t, \cdot)$ conserves the volume measure of $M$ for any $t\ge 0$;
		
		\item[(ii)] there exists a negligible set (with respect to the volume measure) $N\subset M$ such that $t\to X(t,x)$ is absolutely continuous for any $x\in M\setminus N$ and solves \eqref{e:Lagr}.
	\end{itemize}
	Under our assumptions $b\in C^0\cap W^{1,p}$, $\nabla \cdot b=0$, there exists a unique RLF associated to $b$. Uniqueness is understood in the following weak sense: if $X_1$ and $X_2$ are two RLF, then there exists a negligible set $N\subset M$ such that $X_1(t,x) = X_2(t,x)$ for any $x\in M\setminus N$ and $t\ge 0$. 
	
The crucial point in our analysis is to estimate the rate of separation of trajectories, as in the following proposition.

\begin{proposition}[Linear growth of the $H^1$ norm of the flow]\label{p:propagation}
	
	Let $b=\nabla^\perp H \in C^0\cap W^{1,p}(M)$ and $c_0>0$, $\Omega:= H^{-1}((h_0,h_1))$ be as in Lemma \ref{lemma:invariant set}.
	Then there are two constants $r,C>0$ and a function $g \in W^{1,p}_{\mathrm{loc}}(\RR)$ such that for every $\bar z\in \Omega$ and $z \in B_r(\bar z)$ and every $t>0$ it holds
	\begin{equation}
	\dd_M(X(t,z),X(t,\bar z)) \le C(1+t)\big[\dd_M(z,\bar z) + |g(H(z))-g(H(\bar z))|\big] ,
	\end{equation}
    where $X$ is a suitable representative of the unique RLF associated to $b$.
	In particular there is $C'=C'(b,c_0,p)>0$ such that
	\begin{equation}
	\| X(t,\cdot)\|_{W^{1,p}(\Omega)} \le C'(1+t),
	\end{equation}
 for every $t>0$.	
\end{proposition}

The proof of Proposition \ref{p:propagation} follows from the work \cite{Mar21}. For the reader's convenience we outline the main steps. 
By combining \cite{Mar21}*{Lemma 3.2} and \cite{Mar21}*{Remark 3.3}, we get the following result.
	
\begin{lemma}\label{l:imported}
	Let $b= \nabla ^\perp H$, $c_0>0$ and $\Omega:= H^{-1}((h_0,h_1))$ as in Lemma \ref{lemma:invariant set}.
	Then there exist a representative of the regular Lagrangian flow $X$, a function $g \in W^{1,p}_{\rm loc} \cap C^0(\RR)$ and $r>0$ such that for every $t>0$ the following holds:
	there exist $c_1,c_2>0$ such that $\forall \bar z \in \Omega$ and every $z \in B_r(\bar z)$ there exists $s>0$ such that
	\begin{enumerate}
		\item $\dd_M(X(t,\bar z) ,X(s,z)) \le c_1 |H(\bar z) - H(z)|$ \\
		\item $ |t-s| \le c_2 \left( |g(H(\bar z)) - g(H(z))| + \dd_M(\bar z ,z)\right) $.
	\end{enumerate}
\end{lemma}
	
\begin{remark}\label{r:constants}
	The constants $c_1,c_2$ are explicitly chosen at the end of the proof of \cite{Mar21}*{Lemma 3.2} as
	$$
	c_1 = 2(c_0^{-1}+1), \qquad c_2 = \tilde N(c_0^{-1}+1)^2 (1+2 \|b\|_{L^\infty}) + 2(c_0^{-1}+1),
	$$
	where $\tilde N = \left\lceil \frac{t\|b\|_{L^\infty}}{\bar r} \right\rceil$ and $\bar r >0$ is the size of a suitable covering of $\Omega$ depending only on $c_0$ and $H$.
	In particular $c_1$ is independent on $t$ and there is $\tilde c=\tilde c(H,c_0)>0$ such that $c_2 \le \tilde c (1+t)$.
\end{remark}

\begin{remark}
	The function $g$ can be chosen as 
	\begin{equation*}
		g(h) = \int_{\{H\le h\}}|Db|(z)\dd z,
	\end{equation*}
	in particular if $b \in W^{1,p}(\Omega)$, then $g \in W^{1,p}_{\mathrm{loc}}(\RR)$, see \cite{Mar21}.
\end{remark}

\begin{remark}
	The proof of Lemma \ref{l:imported} (see \cite{Mar21}) can be interpreted in terms of the action-angle variables in the smooth setting: the time $s$ in the statement of Lemma \ref{l:imported} can be chosen as the time needed by the trajectory starting at $z$ to run across the same number of periods as the trajectory starting at $\bar z$ at time $t$ and reach the same angular variable of the point $X(t,\bar z)$.
\end{remark}
	
\begin{proof}[Proof of Proposition \ref{p:propagation}]
	The statement immediately follows from Lemma \ref{l:imported} and Remark \ref{r:constants} by estimating
	\begin{align*}
	\dd_M(X(t,z), X(t,\bar z)) &\le \dd_M(X(t,z),X(s, z)) + \dd_M(X(s,z), X(t, \bar z)) \\
	&\le \|b\|_{L^\infty}  |t-s| +  \dd_M(X(s,z) , X(t, \bar z))
	\end{align*}
	and using that $H$ is Lipschitz.
\end{proof}

\begin{proof}[Proof of Proposition \ref{prop:keystep} in the general case]

We appeal to the following fundamental property of RLF \cite{DiPernaLions89}:
\begin{equation}
	\rho(t,x) = \rho_0(X(-t, x)) \, , \quad
	\text{for a.e. $x\in M$}\, .
\end{equation}
In particular, the regularity of $X$ proved in Proposition \ref{p:propagation} immediately implies Proposition \ref{prop:keystep}.
\end{proof}

\section{Mixing, dissipation, and the role of elliptic points}
This section is dedicated to the proofs of the main results of this paper. These concern the general bound on the dissipation rate of Theorem \ref{thm:main},
a sharp treatment of the possible elliptic points in $H$ (cf. Theorem \ref{thm:elliptic}), and a detailed analysis of the mixing properties of a standard cellular flow as in Theorem \ref{thm:cell}.

\subsection{General upper bounds on the dissipation rate}
First we notice that to prove the upper bound \eqref{eq:upperrnu}, it is enough to exhibit an initial datum $\rho_0$ for which the corresponding solution $\rho^\nu$ of  \eqref{e:adv-diff} 
cannot diffuse at a time-scale faster than $O(\nu^{-1/3})$. For this purpose, we will choose $\rho_0\in C^1(M)$ supported in the good invariant set $\Omega$ of Proposition \ref{prop:keystep}.
The idea is to  turn \eqref{eq:reg} into a quantitative vanishing viscosity bound, as stated in the result below.

\begin{lemma}\label{lem:vanish}
	 Let $\nu\in[0,1)$ and  $H\in C^1\cap W^{2,p}(M)$ for some $p\geq 2$.
	 For any  $\rho_0 \in C^1(M)$ with $\spt (\rho_0) \subset \Omega$, there holds the estimate
	\begin{equation}
		\| \rho^\nu(t) - \rho(t) \|_{L^2}^2 \lesssim \nu (1+t)^3 \|\nabla \rho_0\|_{L^\infty}^2  ,
		\qquad \forall t\geq 0,
	\end{equation}
	where $\rho^\nu$, $\rho$ solve \eqref{e:adv-diff} and \eqref{e:transport}, respectively, with the same initial datum $\rho_0$.	
\end{lemma}

\begin{proof}
We take the difference between \eqref{e:adv-diff} and \eqref{e:transport} to obtain
  \begin{equation}\label{e:difference} 
  	\partial_t(\rho^\nu - \rho) + b\cdot \nabla(\rho^\nu - \rho) = \nu \Delta \rho^\nu.
  \end{equation}
  Testing \eqref{e:difference} with $\rho^\nu-\rho$, integrating on $M$, and using the antisymmetry of the transport term, we get
  \begin{align}\label{e:Elg_trick}
  		\ddt \|\rho^\nu  - \rho\|^2_{L^2} 
  		=   - 2\nu \int_M  \nabla \rho^\nu \cdot (\nabla\rho^\nu - \nabla \rho) \dd x 
  		\le    - 2\nu\|\nabla \rho^\nu \|_{L^2} ^2 + 2\nu\|\nabla \rho^\nu \|_{L^2} \|\nabla \rho \|_{L^2}  
  		\le  \nu\|\nabla \rho \|_{L^2}^2.
  \end{align}
 By H\"older inequality and Proposition \ref{prop:keystep} with $p\geq 2$,  we have
  \begin{equation}
  \nu  \|\nabla \rho (t) \|_{L^2}^2
  \lesssim   \nu  \|\nabla \rho(t) \|_{L^p}^2
  \lesssim  \nu  (1+t)^2 \|\nabla \rho_0\|^2_{L^\infty}  .
  \end{equation}
  Plugging this into \eqref{e:Elg_trick} and integrating in time, we obtain
  \begin{equation}
  \|\rho^\nu (t) - \rho(t)\|^2_{L^2} \lesssim \nu  \|\nabla \rho_0\|^2_{L^\infty}    \int_0^t (1+ s)^2 \dd s,
  \end{equation}
  and the proof is over.
\end{proof}
With the above proximity estimate at end, the proof of Theorem \ref{thm:main} follows from a suitable lower bound on the energy dissipation rate.

\begin{proof}[Proof of Theorem \ref{thm:main}]
We take $\rho_0$ as in the above Lemma \ref{lem:vanish}, and we assume that the corresponding solution $\rho^\nu$ to  \eqref{e:adv-diff} experiences
enhanced dissipation at rate $\lambda(\nu)$. According to Definition \ref{def:diff-enhancing}, this implies that
\begin{equation}\label{eq:mainproof1}
	\left(1-A\e^{- \lambda(\nu) t}\right) \norm{\rho_0}_{L^2} \le
	\norm{\rho_0}_{L^2}
	-\norm{\rho^\nu(t)}_{L^2}
	=\norm{\rho(t)}_{L^2}
	-\norm{\rho^\nu(t)}_{L^2}
	  \lesssim \sqrt{\nu (1+t)^3} \|\nabla \rho_0\|_{L^\infty},
\end{equation}
or equivalently
\begin{equation}\label{eq:mainproof2}
	\left(1-C(\Omega,H)\frac{\|\nabla \rho_0\|_{L^\infty}}{\norm{\rho_0}_{L^2}} \sqrt{\nu (1+t)^3}  \right) 
	  \leq A\e^{- \lambda(\nu) t} , \qquad \forall t\geq0.
\end{equation}
Fix now $\eps_0>0$  so that 
\begin{equation}\label{eq:eps0}
\eps_0:= \frac{1}{2} \left[\frac{\| \rho_0\|_{L^2}}{C(\Omega,H)\norm{\nabla\rho_0}_{L^\infty}} \right]^{2/3}
\end{equation}
and $\nu_0\in (0, \eps_0^3)$. Taking $t=\eps_0\nu^{-1/3}\geq 1$ in \eqref{eq:mainproof2}, we end up with
\begin{equation}\label{eq:mainproof3}
	  \frac12\leq A\e^{- \eps_0\lambda(\nu)\nu^{-1/3}} ,  \qquad \forall \nu\in (0,\nu_0),
\end{equation}
which readily implies the bound \eqref{eq:upperrnu} and concludes the proof.
\end{proof}

\subsection{Bounds on mixing and enhanced dissipation rates near elliptic points}

Let $H$, $p=(0,0)$, $r_0$, $\ell_r$ as in Section \ref{subsec:elliptic}.
The first lemma crucially associates  the behavior of $T$ near $r=0$ with the 
growth of the flow map of $b=\nabla^\perp H$.

\begin{lemma}\label{lem:growthellipt}
Let $H\in C^3(M)$ and $(0,0)$ as above. Then there exists a constant $C_1=C_1(H)>0$ such that
\begin{align}\label{eq:growthellipt}
|\nabla X(t,x)|\leq 1+C_1 r |T'(r)| \frac{t}{T(r)},
\end{align} 
for every $r\in (0,r_0)$, $t\geq 1$, and $x\in \ell_r$.
\end{lemma}

\begin{proof}
We preliminary observe that it is enough to estimate the growth of $|\nabla X(t,re_1)|$, as the same estimate will hold
for any point on $\ell_r$. We are interested in the range $t\gg T(r)\sim 1$. 

Let $k\in \mathbb{N}$, and $\delta>0$, we have
\begin{align}
|X(kT(r),re_1)-X(kT(r),(r+\delta)e_1)|
&\leq |X(kT(r),re_1)-X(kT(r+\delta),(r+\delta)e_1)|\\
&\quad+|X(kT(r+\delta),(r+\delta)e_1)-X(kT(r),(r+\delta)e_1)|\\
&\leq \delta + \|b\|_{L^\infty(\ell_r)}k |T(r+\delta)-T(r)|.
\end{align} 
Noting that $\|b\|_{L^\infty(\ell_r)}\lesssim r$, dividing by $\delta$ and sending $\delta\to0$ we find
\begin{align}
|\nabla  X(kT(r),re_1)\cdot e_1|\leq 1 + C_1rk|T'(r)|,
\end{align} 
by choosing $k \sim t/T(r)$ we deduce
\begin{align}
	|\nabla X(t,x) \cdot e_1|\leq 1+C_2 r |T'(r)| \frac{t}{T(r)},
\end{align}
for every $r\in (0,r_0)$, $t\geq 1$, and $x\in \ell_r$.
Here we used that 
$$
| \nabla X(t + T(r),x)\cdot e_1| \le |\nabla X(t,X(T(r),x))\cdot e_1| \e^{CT(r)}
$$
for any $t\in \mathbb{R}$ and that $T(r)\le C$, provided $r\le r_0$.

To estimate $|\nabla X(t,x)|$ it is enough to observe that $re_1$ is almost normal to $\ell_r$ at any $x\in \ell_r$, provided $r\le r_0$ is small enough. On the other hand, the tangential derivative of $X(t,x)$ along $\ell_r$ is bounded by $\|b\|_{L^\infty(\ell_r)}\lesssim r$. This concludes the proof.

\end{proof}

Next we show that an estimate of type \eqref{eq:growthellipt} along with $T'(r)\sim r^\beta$ implies sharper bounds on the mixing rate and
bounds on the enhanced dissipation time-scale for the Hamiltonians considered. 

\begin{lemma}\label{lem:mixdiffelliptic2}
	Let $H\in C^3(M)$ and $(0,0)$ as above. Assume further that the period satisfies
	\begin{align}\label{eq:periodexpansion}
		 T'(s)\sim  s^\beta, \qquad \text{as}\quad s\to 0
	\end{align} 
	for some $\beta \ge 0$.
	Then the corresponding  enhanced dissipation rate of $b$ has the upper upper bound
	\begin{align}\label{eq:endiffelliptic}
		\lambda(\nu)\leq C_2\nu^\frac{1+\beta}{3+\beta},
	\end{align}
	for some constant $C_2=C_2(H)>0$.
\end{lemma}

\begin{proof}
Fix $r\in (0,r_0)$ and consider a smooth and mean-free initial datum $\rho_{0,r}$ supported in an invariant region contained in the annulus 
$B_{ar}(0)\setminus B_r(0)$, where $a>1$ is a constant depending only on $H$. Accordingly, we normalize the initial datum so that  
$\|\rho_{0,r}\|_{L^2} =1$, implying $\|\nabla \rho_{0,r}\|_{L^2} \sim 1/r$, and denote by $\rho^\nu$, $\rho$ the corresponding solutions to 
\eqref{e:adv-diff} and \eqref{e:transport}, respectively, with the same initial datum $\rho_{0,r}$.

By \eqref{eq:periodexpansion} and using that the Lagrangian flow preserves the Lebesgue measure, we can estimate
\begin{align}
  \int_0^t \|\nabla \rho(s) \|^2_{L^2} \dd s  \le   \|\nabla \rho_{0,r}\|_{L^2}^2\int_0^t \sup_{B_{ar}(0) \setminus B_r(0)} |\nabla X(s,\cdot)|^2 \dd s \lesssim     (t/r^2 + r^{2\beta} t^3),
\end{align}
for every $t\geq 1$.
An optimization in $r$ then leads to the integral bound
\begin{align}
\nu \int_0^t \|\nabla \rho(s) \|^2_{L^2} \dd s\lesssim \nu (1+ t) ^\frac{3+\beta}{1+\beta}, \qquad \forall t\geq0.
\end{align}
Now, thanks to \eqref{e:Elg_trick}, the above integral controls the difference between $\rho^\nu$ and $\rho$ in $L^2$. Thus, as we did for 
 \eqref{eq:mainproof1}, we assume that $\rho^\nu$ is enhanced dissipated at rate $\lambda(\nu)$ and obtain the inequality
\begin{equation}\label{eq:mainproof1elliptic}
	1- C(H)\sqrt{  \nu (1+ t) ^\frac{3+\beta}{1+\beta}} 
	  \leq A\e^{- \lambda(\nu) t}  , \qquad \forall t\geq 0.
\end{equation}
Choosing
\begin{equation}\label{eq:eps0elliptic}
\eps_0:= \frac{1}{2} \left[\frac{1}{2C(\Omega,H) } \right]^{2\frac{1+\beta}{3+\beta}},\qquad 0<\nu_0<   \eps_0^\frac{3+\beta}{1+\beta}, \qquad t=\eps_0\nu^{-\frac{1+\beta}{3+\beta}}\geq 1
\end{equation}
eventually leads to
\begin{equation}\label{eq:mainproof3elliptic}
	  \frac12\leq A\e^{- \eps_0\lambda(\nu)\nu^{-\frac{1+\beta}{3+\beta}}} ,  \qquad \forall \nu\in (0,\nu_0),
\end{equation}
which readily implies the bound \eqref{eq:endiffelliptic} and concludes the proof.
\end{proof}

\begin{remark}
In a similar fashion, we can obtain a  lower bound  on the mixing rate for such $b$.
Indeed, let $\rho_0 =\rho_{0,r}$ exactly as in the proof of Lemma \ref{lem:mixdiffelliptic2}. If $b$ is mixing with rate $\gamma(t)$ as in Definition \ref{def:mixing rate}, then we have
\begin{equation}
	\| \rho(t) \|_{H^{-1}} \gtrsim   \frac{\|\rho(t)\|_{L^2}^2}{\| \rho(t)\|_{H^1}}
	\gtrsim \|\rho_{0}\|_{H^1} \frac{1}{1/r^2 + r^{\beta - 1}t} \, .
\end{equation}
Hence $\gamma(t) \gtrsim \frac{1}{1/r^2 + r^{\beta - 1}t}$ for every $r>0$.
We choose $r= t^{-\frac{1}{\beta + 1}}$ and find the lower bound $\gamma(t)\gtrsim \frac{1}{t^{\frac{2}{\beta + 1}}}$.
\end{remark}

\begin{proof}[Proof of Theorem \ref{thm:elliptic}]
	It is enough to apply Lemma \ref{lem:mixdiffelliptic2} and the above remark.	
\end{proof}

\subsection{On the behavior of $T'$ near the elliptic point}
We conclude this section by showing that for every smooth Hamiltonian, $T'$ is bounded near zero (so $\beta\geq 0$ in \eqref{eq:periodexpansion}).
\begin{lemma}
Let $H \in C^3(M)$ and $(0,0)$ be as above. Then $|T'(s)| \lesssim 1$ as $s \to 0$. 
\end{lemma}
\begin{proof}
We have $T(s)= \bar T (H(s e_1))$, where $\bar T(h) =  \int_{\{H=h\}} \frac{1}{|b|} \dd \haus^1$. In particular  
\[ |T'(s)| = |\bar T'(H(se_1)) \frac{\dd}{\dd s}H(se_1)| \sim  |\bar T'(H(se_1))| s,\] therefore it follows from \eqref{eq:period} that it is sufficient to prove 
\begin{equation}\label{e:claim_T'}
\left| \int_{\{H=s^2\}} \nabla\cdot\left(\frac{\nabla H}{|\nabla H|^2}\right) |\nabla H|^{-1} \dd\haus^1\right|\lesssim s^{-1}.
\end{equation}
Observe that a trivial estimate on the size of the integrand would give a useless final upper bound of size $s^{-2}$.
We compare $H$ with the quadratic Hamiltonian $\tilde H (x) = H(0) + x^\perp D^2H(0) x$ to obtain a cancellation at the first non-trivial order: observe that the period of the trajectories associated to $\tilde H$ is constant, in particular 
\[
 \int_{\{\tilde H=s^2\}} \nabla\cdot\left(\frac{\nabla \tilde H}{|\nabla \tilde H|^2}\right) |\nabla \tilde H|^{-1} \dd\haus^1 = 0.
\]
For shortness we denote by $f =  \nabla\cdot\left(\frac{\nabla H}{|\nabla H|^2}\right) |\nabla H|^{-1}$ and $\tilde f =  \nabla\cdot\left(\frac{\nabla \tilde H}{|\nabla \tilde H|^2}\right) |\nabla \tilde H|^{-1}$.
It is straightforward to check that for every $x \in \{H = s^2\} \cup \{\tilde H= s^2\}$ it holds
\[
|f(x)| \lesssim s^{-3},\quad  |\nabla f(x)| \lesssim s^{-4}, \quad  |\tilde f(x)| \lesssim s^{-3}, \quad |\nabla \tilde f(x)| \lesssim s^{-4}.
\]
Let $\tilde \Omega$ be a sufficiently small neighborhood of $(0,0)$ and $g: \tilde \Omega \to \mathbb{R}^2$ be the map $g(x) = tx$ where $t=t(x)$ is the unique value $t>0$ such that $\tilde H(tx) =  H(x)$.

Let $g:\{H = s^2\} \to \{\tilde H = s^2\}$ be the map $g(x) = tx$ where $t=t(s,x)$ is the unique value $t>0$ such that $g(x) \in  \{\tilde H = s^2\}$.
By construction we have 
\[ 
|g- \mathrm{Id}| \lesssim s^3, \quad  |g^{-1}- \mathrm{Id}| \lesssim s^3, \quad |dg- \mathrm{Id}| \lesssim s, \quad  |dg^{-1}- \mathrm{Id}| \lesssim s \quad \mbox{as }s \to 0.
\]
In particular it holds
\begin{align*}
\left| \int_{\{H=s^2\}}f \dd \haus^1 \right|  =&    \left| \int_{\{H=s^2\}}f \dd \haus^1 - \int_{\{\tilde H=s^2\}}\tilde f \dd \haus^1 \right| \\ 
\le &   \int_{\{H=s^2\}}|f-\tilde f| \dd \haus^1 + \left|  \int_{\{ H=s^2\}}\tilde f \dd \haus^1  -  \int_{\{\tilde  H=s^2\}}\tilde f \dd \haus^1 \right|.
\end{align*}
Since on $\{H= s^2\}$ we have the estimates 
\[
|H - \tilde H| \lesssim s^3, \quad |\nabla H - \nabla \tilde H| \lesssim s^2, \quad |\nabla^2H - \nabla^2 \tilde H| \lesssim s, \quad |\nabla H|, |\nabla \tilde H| \sim s
\]
then $|f - \tilde f| \lesssim s^{-2}$. Moreover $\haus^1(\{H = s^2\}) \sim s$, therefore we can estimate 
\[
\int_{\{H=s^2\}}|f-\tilde f| \dd \haus^1\lesssim s^{-1}.
\]
Now we consider
\begin{align*}
 \left|  \int_{\{ H=s^2\}}\tilde f \dd \haus^1  -  \int_{\{\tilde  H=s^2\}}\tilde f \dd \haus^1 \right|  = &  \left|  \int_{\{ H=s^2\}}\tilde f \dd \haus^1  -  \int_{\{H=s^2\}}\tilde f \circ g^{-1} \dd g_\sharp\haus^1\vert_{\{\tilde H=s^2\}}  \right| \\
 \le &   \int_{\{ H=s^2\}}| \tilde f - \tilde f\circ g^{-1}| \dd \haus^1 \\
 & + \left|  \int_{\{ H=s^2\}}\tilde f \circ g^{-1}\dd \haus^1  -  \int_{\{H=s^2\}}\tilde f \circ g^{-1} \dd g_\sharp\haus^1\vert_{\{\tilde H=s^2\}}  \right|.
\end{align*}
Since $|\nabla \tilde f| \lesssim s^{-4}$ and $ |g^{-1}- \mathrm{Id}| \lesssim s^3$, we can estimate
\[
 \int_{\{ H=s^2\}}| \tilde f - \tilde f\circ g^{-1}| \dd \haus^1\lesssim 1.
\]
Moreover, by the estimates on $dg$ we obtain that $g_\sharp \haus^1\vert_{\{\tilde H=s^2\}} = \psi  \haus^1\vert_{\{ H=s^2\}}$ with $|\psi - 1| \lesssim s$.
Since $|\tilde f \circ g^{-1}|\lesssim s^{-3}$, then we deduce that 
\[
 \left|  \int_{\{ H=s^2\}}\tilde f \circ g^{-1}\dd \haus^1  -  \int_{\{H=s^2\}}\tilde f \circ g^{-1} \dd g_\sharp\haus^1\vert_{\{\tilde H=s^2\}}  \right|  \lesssim s^{-1}.
\]
Combining the previous estimates we obtain \eqref{e:claim_T'} and therefore the claim.
\end{proof}

\section{The case of a cellular flow}

We consider the $2\pi$-periodic Hamiltonian $H_\mathsf{c}(x)=H_\mathsf{c}(x_1, x_2) = \sin x_1 \, \sin x_2$ and the associated velocity field
\begin{equation}
	b_\sfc(x) = b_\sfc(x_1,x_2) = (-\sin x_1\, \cos x_2, \cos x_1\, \sin x_2)\, .
\end{equation}
We restrict our analysis to the domain $[0,\pi]^2$ and observe that $H_\mathsf{c}\ge 0$ on $[0,\pi]$, $H_\mathsf{c}=0$ at $\partial [0,\pi]^2$ and $(\pi/2,\pi/2)$ is the only maximum point.
For any $h\in (0,1)$ we denote by $T(h)$ the period of the closed trajectory supported in $\{H_\mathsf{c}=h\}\cap [0,\pi]^2$.

\begin{lemma}
	For any $h\in (0,1)$ we have
	\begin{equation}
		2h(1-h) \le |b_\sfc(x)|^2 \le 2(1-h^2) \, ,
		\quad 
		\text{for any $x\in \{H_\mathsf{c}=h\}\cap [0,\pi]^2$}\, .
	\end{equation}
\end{lemma}

\begin{proof}
	Fix $x\in \{H_\mathsf{c}=h\}\cap [0,\pi]^2$, and compute
	\begin{equation}
		|b_\sfc(x)|^2 = - 2 (\sin x_1\, \sin x_2)^2 + (\sin x_1)^2 + (\sin x_2)^2
		= -2 h^2 + (\sin x_1)^2 + (\sin x_2)^2.
	\end{equation}
	This yields
	\begin{equation}
		-2 h^2 + (\sin x_1)^2 + (\sin x_2)^2 \le -2 h^2 + 2 = 2(1-h^2) ,
	\end{equation}
	and
	\begin{equation}
		-2 h^2 + (\sin x_1)^2 + (\sin x_2)^2 = -2 h^2 + 2h + (\sin x_1 - \sin x_2)^2 \ge 2h(1-h),
	\end{equation}
	as we wanted.
\end{proof}

\subsection{Estimates on $T(h)$} 
We provide an explicit formula for $T(h)$ and we show that $T(h)\sim1+ \ln(1/h) $ up to the second order.

\begin{lemma}
	For any $h\in (0,1)$ we have
	\begin{equation}\label{eq:per cellular}
		T(h) = 4\int_h^1 \frac{1}{\sqrt{x^2-h^2}} \frac{1}{\sqrt{1-x^2}} \dd x 
		=
		4
		\int_0^1 \frac{1}{\sqrt{1-x^2}} \frac{1}{\sqrt{1-(1-h^2)x^2}} \dd x 
		\, .
	\end{equation}
\end{lemma}

\begin{proof}
	We parametrize a quarter of the closed curve $\{H_\mathsf{c}=h\}\cap [0,\pi]^2$ with a curve $[\arcsin h, \pi/2]\ni t\to (t, \gamma(t))$, where $\sin t\, \sin \gamma(t) = h$.
	It is not hard to check that
	\begin{equation}\label{z3}
		|\gamma'(t)|^2 + 1 = \frac{h^2 - 2h^2\sin^2 t + \sin^4 t}{\sin^2 t (\sin^2 t - h^2)}\, .
	\end{equation}
	Moreover, from the identity $|b_\sfc(x_1,x_2)|^2 = \sin^2 x_1 \, \cos^2 x_2 + \cos^2 x_1\, \sin^2 x_2$ we get
	\begin{equation}\label{z4}
		|b_\sfc(t,\gamma(t))|^2 = \frac{h^2 - 2h^2\sin^2 t + \sin^4 t}{\sin^2 t}\, .
	\end{equation}
	By combining \eqref{z3} and \eqref{z4} we can compute the period
	\begin{equation}
		T(h) 
		= 4\int_{\arcsin h}^{\pi/2} \frac{\sqrt{|\gamma'(t)|^2 + 1}}{|b_\sfc(t,\gamma(t))|}  \dd t
		= 4\int_{\arcsin h}^{\pi/2} \frac{1}{\sqrt{\sin^2 t - h^2}}  \dd t,
	\end{equation}
	and the first identity in \eqref{eq:per cellular} follows from the change of variables $x = \sin t$.
	To prove the second identity in \eqref{eq:per cellular} we employ the change of variables $y = \sqrt{\frac{1-x^2}{1-h^2}}$.
\end{proof}

\begin{lemma}\label{l:estimates_T}
	The function $T$ is smooth in $(0,1)$, and there exists $C>1$ such 
	\begin{align}
		\frac{T(h)}{1 + \ln(1/h)} + h |T'(h)| + h^2 |T''(h)|
		& \le C, \label{est1 on T}
		\\
		-T'(h) &\ge \frac{1}{C h} , \label{est2 on T}
	\end{align}
	for any $h\in (0,1)$.
\end{lemma}

\begin{proof}
	From the second identity in \eqref{eq:per cellular} it is immediate to verify that $T\in C^\infty((0,1))$.
	Let us estimate $T(h)$. We first consider the case $h\in (1/2,1)$:
	\begin{align}
		T(h) & \le 10 \int_h^1  \frac{1}{\sqrt{x-h}\sqrt{1-x}} \dd x
		\\& = 10  \int_h^{\frac{h+1}{2}}  \frac{1}{\sqrt{x-h}\sqrt{1-x}} \dd x
		+ 10\int_{\frac{h+1}{2}}^1 \frac{1}{\sqrt{x-h}\sqrt{1-x}} \dd x
		\\& \le \frac{10\sqrt{2}}{\sqrt{1-h}}  \int_h^{\frac{h+1}{2}}  \frac{1}{\sqrt{x-h}} \dd x
		+
		\frac{10\sqrt{2}}{\sqrt{1 - h  }}\int_{\frac{h+1}{2}}^1 \frac{1}{\sqrt{1-x}} \dd x
		\\& = 40 .
	\end{align}
	Let us now assume that $h\in (0,1/2)$, we have
	\begin{align}
		T(h) & \le 4 \int_h^1 \frac{1}{\sqrt{x^2 - h^2}} \dd x
		= 4 \int_1^{1/h} \frac{1}{\sqrt{z^2 - 1}}\dd z
		\\& = 4 \int_1^{2} \frac{1}{\sqrt{z^2 - 1}}\dd z+ 4\int_2^{1/h} \frac{1}{\sqrt{z^2 - 1}} \dd z
		\\& \le 10 + 10 \ln(1/h)\, ,
	\end{align}
	where we used the change of variables $z=x/h$.
	
	To study $T'(h)$, we differentiate the second identity in \eqref{eq:per cellular} and obtain
	\begin{equation}
		-T'(h) = 4 \int_0^1 \frac{1}{(1-x^2)^{1/2}}\, \frac{h x^2}{(1-(1-h^2)x^2)^{3/2}} \dd x . 
	\end{equation}
	We can split the integral into
	\begin{align}
		- h T'(h)  & 
		= 4 \int_0^{1-h^2} \frac{1}{(1-x^2)^{1/2}}\, \frac{h^2 x^2}{(1-(1-h^2)x^2)^{3/2}} \dd x\notag
		\\ & \quad 
		+4 \int_{1-h^2}^1 \frac{1}{(1-x^2)^{1/2}}\, \frac{h^2 x^2}{(1-(1-h^2)x^2)^{3/2}} \dd x\notag
		\\&
		=: I + II \, ,
	\end{align}
	and estimate
	\begin{align}
		I &\le \frac{4h^2(1-h^2)}{(1-(1-h^2)^2)^{1/2}} \int_0^{1-h^2} \frac{x}{(1-(1-h^2)x^2)^{3/2}} \dd x\notag
		\\& 
		\le 4 h \int_0^{1-h^2} \frac{\dd}{\dd x} (1-(1-h^2)x^2)^{-1/2} \dd x \notag
		\\& \le 4.
	\end{align}
	From
	\begin{equation}
		\frac{4 h^2(1-h^2)}{(1-(1-h^2)^3)^{3/2}} \int_{1-h^2}^1 \frac{x}{(1-x^2)^{1/2}}\dd x 
		\le II
		\le 
		\frac{4}{h}
		\int_{1-h^2}^1 \frac{x}{(1-x^2)^{1/2}}\dd x
	\end{equation}
	we deduce that
	\begin{equation}
		1-h \le II \le 8.
	\end{equation}
	Since $I\geq 0$, $T'$ is continuous in $h=1$ and $-T'(1)=\pi$ we also deduce \eqref{est2 on T}. 
	
	To conclude the proof we have to show the upper bound on $h^2|T''(h)|$. Let us introduce the auxilary function
	\begin{equation}
		G(s) := 	4
		\int_0^1 \frac{1}{\sqrt{1-x^2}} \frac{1}{\sqrt{1-sx^2}}\, \dd x\qquad
		s\in (0,1) .
	\end{equation}
	Observe that $T(h) = G(1-h^2)$, hence
	\begin{equation}
		h^2 T''(h) = -2 h^2 G'(1-h^2) + 4 h^4 G''(1-h^2)
		= h T'(h) + 4 h^4 G''(1-h^2)\, .
	\end{equation}
	We already know that $h|T'(h)| \le 12$. To estimate the second term we use the identity
	\begin{align}
		h^4 |G''(1-h^2)| 
		& = 3 \int_0^1 \frac{1}{(1-x^2)^{1/2}}\, \frac{h^4 x^4}{(1-(1-h^2)x^2)^{\frac{5}{2}}} \dd x \notag
		\\& 
		= 3 \int_0^{1-h^2} \frac{1}{(1-x^2)^{1/2}}\, \frac{h^4 x^4}{(1-(1-h^2)x^2)^{\frac{5}{2}}} \dd x\notag
		\\ &\quad  + 3 \int_{1-h^2}^1 \frac{1}{(1-x^2)^{1/2}}\, \frac{h^4 x^4}{(1-(1-h^2)x^2)^{\frac{5}{2}}} \dd x\notag
		\\& =: I + II
	\end{align}
	and estimate $I$ and $II$ exactly as we did for $-hT'(h)$.
\end{proof}

\begin{remark}
The expression for $T(h)$ in \eqref{eq:per cellular} is in fact a complete elliptic integral of the first kind. Its asymptotics are well-known \cite{HandMath}*{Ch. 17}:
as $ h\to 0^+$ we have
\begin{align}
T(h)\sim 4\ln (4/h)+ O(h), \qquad T'(h)\sim -\frac{4}{h}+O(1) \qquad T''(h)\sim \frac{4}{h^2} +O(1/h),
\end{align}
while as $h\to 1^-$ it holds
\begin{align}
T(h)\sim 2 \pi+O (h-1), \qquad T'(h)\sim -\pi+O(h-1) \qquad T''(h)\sim \frac{5\pi}{4}+O (h-1).
\end{align}
However, all we need are the global bounds contained in Lemma \ref{l:estimates_T}.
\end{remark}

\subsection{Reparametrization of the Hamiltionian and regularity of the change of variables}
We consider the coordinates
\begin{equation}\label{eq:changephitilde}
	\tilde\Phi : \mathbb{S}^1 \times \left(0,\frac\pi2\right) \to (0,\pi)^2\setminus \left\{\left( \frac\pi2,\frac\pi2\right)\right\} \, ,
	\quad
	\tilde \Phi(\theta, I) := X\left(\theta \tilde T(I), \left(\frac\pi2, I\right)\right)
\end{equation}
where $\tilde T(I)$ is the period of the closed orbit in $\{H_\mathsf{c} =  \sin I\}$ passing through the point $ \left(\pi/2, I\right)$.
These coordinates are related to the action-angle coordinates introduced in Section \ref{ss:action-angle} by
\begin{equation}\label{e:new_old}
	\tilde \Phi(\theta, I) = \Phi(\theta, \sin I), \qquad \tilde T(I) = T(\sin I).
\end{equation}

\begin{lemma}\label{l:est_action_angle_new}
	The map $\tilde \Phi : \mathbb{S}^1 \times \left(0,\frac\pi2\right) \to (0,\pi)^2\setminus \left\{\left(\pi/2,\pi/2\right)\right\}$ and its inverse are $C^1$. Moreover the following estimates hold:
	\[
	|\partial_\theta \tilde \Phi(\theta, I)| \lesssim |\ln I| \left(\frac\pi2-I\right), \qquad |\partial_I \tilde \Phi(\theta, I)| \lesssim I^{-1}.
	\]
\end{lemma}
\begin{proof}
	The $C^1$-regularity of $\tilde \Phi$ and its inverse follows from \eqref{e:new_old} and the same properties proved for $\Phi$ in Lemma \ref{l:Phi}.
	The first inequality follows from  Lemma \ref{l:estimates_T}:
	\[
	|\partial_\theta \tilde \Phi(\theta, I)| \le |\tilde T(I)| |b_\sfc|(\tilde \Phi(\theta,I)) \lesssim  |\ln I| \left(\frac\pi2-I\right).
	\]
	Let us prove the second inequality: for every $\alpha \in \mathbb{S}^1$ and $I\in \left(0,\frac\pi2\right)$, let $P(\alpha,I) \in \tilde \Phi(\mathbb{S}^1, I)$ and $g(\alpha,I)>0$ be such that
	\begin{equation}\label{e:def_galpha}
		P(\alpha,I) = \left( \frac\pi2,\frac\pi2\right) + g(\alpha,I) \e^{i\left( 3\pi/2- \alpha\right)}.
	\end{equation}
Take now $\alpha :  \mathbb{S}^1 \times \left(0,\pi/2\right) \to \mathbb{S}^1$   such that
	\[
	\e^{i\left( 3\pi/2- \alpha(\theta,I)\right)} = \frac{\tilde \Phi (\theta, I) - \left(\pi/2,\pi/2\right) }{\left| \tilde \Phi (\theta, I) - \left(\pi/2,\pi/2\right)\right|}.
	\]
		In particular we have 
	\[
	\tilde \Phi(\theta, I) = \left(\frac\pi2,\frac\pi2\right) + g(\alpha(\theta,I),I) \e^{i\left(3\pi/2-\alpha(\theta,I)\right)},
	\]
	therefore
	\begin{equation}\label{e:Phi_I}
		|\partial_I\tilde \Phi(\theta,I)| \le |\partial_\alpha g ( \alpha(\theta,I),I)| |\partial_I\alpha(\theta,I)| + |\partial_Ig(\alpha(\theta,I),I)| + |g ( \alpha(\theta,I),I)| |\partial_I\alpha(\theta,I)|.
	\end{equation}
	We observe that $g(\alpha,I) \sim \left(\pi/2-I\right)$. Moreover it holds 
	\begin{align}
	 &|\partial_Ig(\alpha,I)| \lesssim 1,  \qquad \text{as}\quad I\to \frac\pi2^-,\\
	 &|\partial_Ig(\alpha,I)| \lesssim \frac{1}{|b_\sfc\circ P|}\lesssim \frac1{\sqrt I}, \qquad \text{as}\quad I\to 0^+.
	\end{align}
	Differentiating \eqref{e:def_galpha} and observing that $g(\alpha,I), |\partial_\alpha P| \sim \left(\pi/2-I\right)$ we obtain
	$|\partial_\alpha g| \lesssim \left(\frac\pi2-I\right).$
	Therefore it follows by \eqref{e:Phi_I} that
	\begin{equation}\label{e:Phi_I2}
		|\partial_I\tilde \Phi(\theta,I)| \lesssim \left(\frac\pi2-I\right) |\partial_I\alpha(\theta,I)| +  \frac1{\sqrt I}.
	\end{equation}
	In remains to estimate $\partial_I \alpha$: we consider the angular velocity with respect to the center $\left( \pi/2,\pi/2\right)$:
	\[
	\omega : \mathbb{S}^1 \times \left( 0,\frac\pi2\right) \to \RR, \qquad \omega(\alpha,I) = \frac{b_\sfc(P(\alpha,I)) \, \e^{i\left(2\pi - \alpha\right)}}{g(\alpha,I)}.
	\]
	By construction it holds
	\[
	\int_0^{\alpha(\theta,I)} \frac{1}{\omega(\alpha,I)}\dd\alpha = \theta \tilde T(I).
	\]
	Differentiating the expression above with respect to $I$ we get
	\begin{align}\label{eq:dialpha}
	\partial_I\alpha (\theta,I) = \omega(\alpha(\theta,I),I) \left[\theta \tilde T'(I) + \int_0^{\alpha(\theta, I)} \frac{\partial_I\omega(\alpha,I)}{\omega(\alpha,I)^2} \dd\alpha \right]
	\end{align}
	Since $g(\alpha,I) \sim \left(\frac\pi2 - I\right)$ and $b_\sfc(P(\alpha,I)) \,  \e^{i\left(2\pi - \alpha\right)} \ge \sqrt 2/2$, then $\omega(\alpha,I) \sim \frac{|b_\sfc\circ P|}{g(\alpha,I) }\lesssim 1$.
	By Lemma \ref{l:estimates_T} we have $\tilde T'(I) \lesssim I^{-1}\left(\frac\pi2-I\right)$.
	In order to estimate the integral, we distinguish two regimes: $I\to 0^+$ and $I \to  \frac\pi2^-$, and rely on the estimate
	\[
	|\partial_I\omega| \le \frac{|\partial_I (b_\sfc \circ P)|}{g} + \frac{|b_\sfc\circ P| |\partial_Ig|}{g^2}.
	\]
	For $I \to \frac\pi2^-$ we have $g(\alpha,I) \sim \left(\frac\pi2 -I\right)$ and  $ \left(\frac\pi2 -I\right)^{-1} |b_\sfc \circ P| + |\partial_I (b_\sfc\circ P)| + |\partial_Ig| \lesssim 1$, therefore
	\[
	|\partial_I \omega| \lesssim \left(\frac\pi2 -I\right)^{-1}  \qquad \mbox{as } I \to \frac\pi2^-.
	\]
	In particular we get 
	\[
	|\partial_I\alpha (\theta,I)| \lesssim \left(\frac\pi2 -I\right)^{-1}  \qquad \mbox{as } I \to \frac\pi2^-.
	\]
	We now estimate the integral term in \eqref{eq:dialpha} as $I\to 0^+$: by the symmetries of the cellular flow we have that for every $\theta \in \mathbb{S}^1$ it holds
	\[
	\left| \int_0^{\alpha(\theta, I)} \frac{\partial_I\omega(\alpha,I)}{\omega(\alpha,I)^2} \dd\alpha \right| \le  \int_0^{2\pi}\left| \frac{\partial_I\omega(\alpha,I)}{\omega(\alpha,I)^2} \right|\dd\alpha = 8  \int_0^{\frac\pi4} \left|\frac{\partial_I\omega(\alpha,I)}{\omega(\alpha,I)^2}\right| \dd\alpha 
	\]
	If $\alpha, I \in \left(0,\pi/4\right)$, then, denoting by $P_1$ the first component of the vector $P$, we have
	\[
	|b_\sfc\circ P|\sim P_1(\alpha,I), \qquad |\partial_I (b_\sfc\circ P)| \sim \frac{|(D_x b_\sfc)\circ P|}{|b_\sfc\circ P|} \lesssim \frac1{|b_\sfc\circ P|}, \qquad |\partial_Ig| \lesssim \frac1{|b_\sfc\circ P|}, \qquad g \sim 1,
	\]
	therefore $| \partial_I\omega (\alpha,I)| \lesssim \frac1{P_1(\alpha,I)}$ and $|\omega(\alpha,I)| \sim P_1(\alpha,I)$.
	We observe that for every $I \in \left(0,\pi/4\right)$ the map $\alpha \mapsto P_1(\alpha,I)$ is bi-Lipschitz from $\left(0,\pi/4\right)$ to its image with constants independent of $I$. 
	
	Eventually we can estimate
	\[
	\left|\int_0^{\alpha(\theta, I)} \frac{\partial_I\omega(\alpha,I)}{\omega(\alpha,I)^2} \dd\alpha\right| \lesssim \left| \int_{P_1\left( \frac\pi4,I\right)}^{\frac\pi2}  \frac1{x_1^3}\dd x_1\right| \lesssim I^{-1} \qquad \mbox{as }I \to 0^+,
	\]
	where in the last inequality we used that $P_1\left( \frac\pi4,I\right) \lesssim \sqrt I$, since $\sin^2\left( P_1\left( \frac\pi4,I\right) \right) =\sin I$.
	Combining the estimates in the two regimes we get
	\[
	|\partial_I\alpha (\theta,I)| \lesssim  \left(\frac\pi2 -I\right)^{-1} I^{-1},
	\]
	therefore we conclude by \eqref{e:Phi_I2} that 
	\[
	|\partial_I\tilde \Phi (\theta,I)| \lesssim I^{-1},
	\]
	finishing the proof.
\end{proof}

\subsection{Global mixing estimate}\label{sub:globalmix}

Let $\rho_0 \in C^1(\TT^2)$ be mean zero along the level sets of $H_\mathsf{c}$, and consider $\rho : \RR \times \TT^2 \to \RR$  the solution to \eqref{e:transport} with $\rho(0,\cdot) = \rho_0$. 
To prove global mixing estimates, it is enough to study the flow in each invariant domain $(0,\pi)^2$, $(0,\pi)\times (-\pi,0)$, $(-\pi,0)^2$, $(-\pi,0)\times (0,\pi)$. More precisely, we split
	\begin{equation}\label{zz1}
	\begin{split}
		\| \rho(t) \|_{H^{-1}(\TT^2)} 
		&= 
		\sup_{\|\varphi\|_{H^1(\TT^2)}\le 1} \int_{\TT^2} \rho(t) \varphi \dd x
		\\& \le 
		\sup_{\|\varphi\|_{H^1(\TT^2)}\le 1} \int_{(0,\pi)^2} \rho(t) \varphi \dd x + \sup_{\|\varphi\|_{H^1(\TT^2)}\le 1} \int_{(0,\pi)\times (-\pi,0)} \rho(t) \varphi \dd x 
		\\& \quad +
		\sup_{\|\varphi\|_{H^1(\TT^2)}\le 1} \int_{(-\pi,0)\times (0,\pi)} \rho(t) \varphi \dd x
		+
		\sup_{\|\varphi\|_{H^1(\TT^2)}\le 1} \int_{(-\pi,0)^2} \rho(t) \varphi \dd x
	\end{split}
    \end{equation}
and estimate each term separately using the special coordinates introduced in the previous section. We illustrate in details how to estimate the mixing rate in $(0,\pi)^2$, the analysis of the other domains being completely analogous.

We apply the stationary phase argument for the dynamics in the coordinates $(\theta,I)$ and the change of variables $\tilde \Phi$ given in \eqref{eq:changephitilde}.
The function $f(t,\theta,I) := \rho(t,\tilde \Phi(\theta,I))$ solves the equation
\[
\partial_t f + \frac1{\tilde T(I)}\partial_\theta f =0\, ,
\quad 
\text{for } (t,\theta,I) \in \RR \times \mathbb{S}^1\times \left[0,\frac{\pi}{2} \right] .
\] 
By \eqref{e:Jacobian_angle-action} and \eqref{e:new_old} we have
\[
J_{\tilde \Phi}(\theta,I) = J_{\Phi}(\theta,I)\cos I = \tilde T(I)\cos I.
\]
We set $g(I):= \tilde T(I)\cos I$, and work in the weighted Sobolev space $H^1_g$, defined by the norm
\begin{equation}
	\| f \|_{H^1_g}^2 : = \int_{\mathbb{S}^1\times \left(0,\frac\pi2\right)} \left(|f(\theta,s)|^2 + |\partial_s f(\theta,s)|^2\right) g(s) \dd s \dd\theta .
\end{equation}
Fix $\delta, \delta' \in (0,1)$, by \eqref{e:mix_stat_phase} we have
\[
\sup_{\|\phi\|_{H^1_g}\le 1}\int_{\delta}^{\frac\pi2-\delta'} \int_{\mathbb{S}^1} f(t,\theta,s) \phi(\theta,s) g(s) \dd\theta \dd s \le \|f(0,\cdot)\|_{H^1_g}r(t)
\]
with $r(t)$ given by \eqref{eq:rt}, i.e. 
\begin{align}
		r(t) &= 
		\frac{1}{t} \left\| g^{-\frac12}\right\|^2_{L^2(\delta,\frac{\pi}{2}-\delta')} 
		\left(  \frac{\tilde T^2g}{|\tilde T'|}(\delta) +\frac{\tilde T^2g}{|\tilde T'|}(\frac{\pi}{2}-\delta') +  \left\| \left( \frac{\tilde T^2g}{\tilde T'}\right)' \right\|_{L^1(\delta,\frac{\pi}{2}-\delta')} \right)
		\\&
		\quad+ 
		\frac{1}{t} \left\| g^{-\frac12}\right\|_{L^2(\delta, \frac{\pi}{2}-\delta')}\left\| \frac{\tilde T^2 g^{\frac{1}{2}}}{\tilde T'}\right\|_{L^2(\delta,\frac{\pi}{2}-\delta')}\, .\label{eq:rt1}
\end{align}
By Lemma \ref{l:estimates_T} we have
\[
\tilde T(I) \sim 1 + |\ln I|, \qquad |\tilde T'(I)| \sim I^{-1} \left(\frac\pi2 -I\right), \qquad |\tilde T''(I)| \sim I^{-2} \, ,
\]
hence each term in \eqref{eq:rt1} can be estimated pointwise as
\begin{align}
\left| \left( \frac{\tilde T^2g}{\tilde T'}\right)' \right|
 =\left| \left( \frac{\tilde T^3 \cos I}{\tilde T'}\right)' \right|
&\lesssim \tilde T^2  + \frac{(1 + |\ln I|)^3I^2}{\frac{\pi}{2} - I} +\frac{|\tilde T''| \tilde T^3 I^2}{\frac{\pi}{2} - I}
\lesssim \frac{(1+|\ln I|)^3}{\frac{\pi}{2} - I},
\\
\left| \frac{\tilde T^2 g}{\tilde T'}  \right| 
& \lesssim (1+|\ln I|)^3 I,
\\
\left| \frac{\tilde T^2 g^{1/2}}{\tilde T'}  \right| 
& \lesssim \frac{(1+|\ln I|)^{5/2}I}{\left(\frac{\pi}{2} - I\right)^{1/2}}.
\end{align}
Plugging these estimates in \eqref{eq:rt}, we get 
\begin{equation}\label{e:est_r2}
	r(t) \lesssim \left|\ln \delta' \right|^2 
	\frac1t\, .
\end{equation}
Let $\Omega \subset (0,\pi)^2$ be an open set such that $\spt(\rho(t)) \cap (0,\pi)^2 \subset \Omega$ for every $t\in \mathbb{R}$. Denoting by $\cL^2$ the Lebesgue measure on $(0,\pi)^2$, we estimate
\begin{align*}
	& \sup_{\|\varphi\|_{H^1}\le 1} \int_{(0,\pi)^2} \rho(t) \varphi \dd x \\
	 &=   
	\sup_{\|\varphi\|_{H^1}\le 1}\left( \int_{\tilde \Phi (\mathbb{S}^1 \times (0,\delta))} \rho(t) \varphi \dd x 
	+
	\int_{\tilde \Phi (\mathbb{S}^1 \times (\delta,\frac\pi2-\delta'))} \rho(t) \varphi \dd x 
	+
	\int_{\tilde \Phi (\mathbb{S}^1 \times (\frac\pi2-\delta',\frac\pi2))} \rho(t) \varphi \dd x\right) 
	\\
	 &\lesssim  C(\eps) \|\rho_0\|_{L^\infty} \left[\cL^2\left(\left(\tilde \Phi (\mathbb{S}^1 \times (0,\delta)) \cup\tilde \Phi \left(\mathbb{S}^1 \times \left(\frac\pi2-\delta',\frac\pi2\right)\right) \right)\cap\Omega\right)  \right]^{1-\eps}
	\\
	& \qquad + 
	 \sup_{\|\varphi\|_{H^1}\le 1}\int_{\tilde \Phi (\mathbb{S}^1 \times (\delta,\frac\pi2-\delta'))} \rho(t) \varphi \dd x,
\end{align*}
where we used the a priori estimate $\|\rho(t)\|_{L^\infty} \le \|\rho_0\|_{L^\infty}$
and the two-dimensional Sobolev embedding $\|\varphi\|_{L^p}\le C(p) \|\varphi\|_{H^1} \le C(p)$ for any $p<\infty$.
Performing the computations in Section \ref{ss:change_of_variables}, we deduce from \eqref{e:est_r2} and Lemma \ref{l:est_action_angle_new} that 
\begin{align}
	\sup_{\|\varphi\|_{H^1}\le 1}\int_{\tilde \Phi (\mathbb{S}^1 \times (\delta,\frac\pi2-\delta'))} \rho(t) \varphi \dd x 
	\lesssim \frac{\|\rho_0\|_{H^1} }{t}( 1 + \mathrm{Lip}\tilde \Phi)^2\left|\ln \delta'\right|^2
	\lesssim  \frac{\|\rho_0\|_{H^1}}{t} \delta^{-2}\left|\ln \delta'\right|^2 \, .
\end{align}
 Let $\kappa<1/5$ be a small parameter. We claim that for every $\eps>0$ we have the following:
\begin{enumerate} [label=(\arabic*), ref=(\arabic*)]
	\item \label{item:(1)}  If $\spt(\rho_0)\cap (0,\pi)^2 \subset \mathbb{S}^1\times(\kappa, \pi-\kappa)$, then
	\begin{equation}\label{(1)}
		 \sup_{\|\varphi\|_{H^1}\le 1} \int_{(0,\pi)^2} \rho(t) \varphi \dd x \le C(\eps,\kappa, \|\rho_0\|_{H^1},\|\rho_0\|_{L^\infty}) \frac{1}{t^{1-\eps}}\, .
	\end{equation}

	\item\label{item:(2)} If $\spt(\rho_0)\cap (0,\pi)^2 \subset \mathbb{S}^1\times(0,\pi)^2\setminus B_\kappa((\pi/2,\pi/2))$, then
	\begin{equation}\label{(2)}
		 \sup_{\|\varphi\|_{H^1}\le 1} \int_{(0,\pi)^2} \rho(t) \varphi \dd x \le C(\eps,\kappa, \|\rho_0\|_{H^1},\|\rho_0\|_{L^\infty})   \frac{1}{t^{\frac{1}{3}-\eps}}\, .
	\end{equation} 
\end{enumerate}
To prove \ref{item:(1)}, we set $\Omega:=  \mathbb{S}^1\times(\kappa', \pi-\kappa')$ for some $0<\kappa'<\kappa$ in such a way that
 it contains $\spt(\rho(t))\cap (0,\pi)^2$ for any $t\in \mathbb{R}$. Hence, there exists $\delta(\kappa)<1/5$, depending only on $\kappa$, such that $\tilde \Phi \left(\mathbb{S}^1 \times (0,\delta(\kappa))\right)\cap \Omega = \emptyset$. The analysis above with $\delta = \delta(\kappa)$ gives
\begin{equation}
	 \sup_{\|\varphi\|_{H^1}\le 1} \int_{(0,\pi)^2} \rho(t) \varphi \dd x
	 \lesssim C(\eps)\|\rho_0\|_{L^\infty} (\delta')^{2-2\eps} + C(\kappa) \|\rho_0\|_{H^1} |\ln \delta'|^2 \frac{1}{t} \, ,
\end{equation}
where we used that
\begin{equation}
	\cL^2\left(\tilde \Phi \left(\mathbb{S}^1 \times \left(\frac\pi2-\delta',\frac\pi2\right)\right)\right) \lesssim (\delta')^2 \, .
\end{equation}
The sought conclusion comes by optimizing $\delta'$.

To prove \ref{item:(2)} we argue analogously. The right domain to consider is 
$$
\Omega:=(0,\pi)^2 \subset (0,\pi)^2\setminus B_{2\kappa}((\pi/2,\pi/2)),
$$
and we use the volume estimate
 \begin{equation}
 	\cL^2\left(\tilde \Phi (\mathbb{S}^1 \times (0,\delta))\right) \lesssim \delta ( 1 + |\ln \delta|) \, .
 \end{equation}

 \begin{proof}[Proof of Theorem \ref{thm:cell}]
 
 Let $\rho_0 \in C^1(\TT^2)$ be mean zero along the level sets of $H_\mathsf{c}$, and consider $\rho : \RR \times \TT^2 \to \RR$  the solution to \eqref{e:transport} with $\rho(0,\cdot) = \rho_0$. 
 The lower bound on the mixing rate in \eqref{eq:mixingcell} and the upper bound on the enhanced dissipation rate in \eqref{eq:enhancedcell} are a consequence of Theorem \ref{thm:elliptic}.
 We then prove the upper bound on the mixing rate in \eqref{eq:mixingcell} .
 
Let $\chi$ be a smooth function, constant along the levels of $H_\mathsf{c}$ such that $\chi= 1$ in $B_{\frac{1}{100}}((\pi/2,\pi/2))$ and $\chi = 0 $ in the complement of $B_{\frac{1}{10}}((\pi/2,\pi/2))$. It turns out that $\chi \rho(t)$ and $(1-\chi)\rho(t)$ both solve \eqref{e:transport}. Hence, we can apply \eqref{(1)} and \eqref{(2)} to get
 \begin{align}
 	 \sup_{\|\varphi\|_{H^1}\le 1} \int_{(0,\pi)^2} \rho(t) \varphi \dd x
 	& \le  \sup_{\|\varphi\|_{H^1}\le 1} \int_{(0,\pi)^2} \chi\rho(t) \varphi \dd x +  \sup_{\|\varphi\|_{H^1}\le 1} \int_{(0,\pi)^2} (1-\chi)\rho(t) \varphi \dd x
 	\\& \le C(\eps,\kappa, \| \chi \rho_0\|_{H^1}, \| \chi \rho_0\|_{L^\infty}) \frac{1}{t^{1-\eps}} 
 	\\& \quad + C(\eps,\kappa, \| (1-\chi) \rho_0\|_{H^1}, \|(1- \chi) \rho_0\|_{L^\infty})  \frac{1}{t^{\frac{1}{3}-\eps}}
 	\\&
 	\le C(\eps, \kappa, \|\rho_0\|_{H^1},\|\rho_0\|_{L^\infty}) \frac{1}{t^{\frac{1}{3}-\eps}}\, .
 \end{align}
 To estimate $\| \rho(t) \|_{H^{-1}(\TT^2)}$ we use the decomposition \eqref{zz1}, and sum all the contributions.
 
 Now,  the lower on the enhanced dissipation rate in \eqref{eq:enhancedcell} is a consequence of \cite{CZDE20}*{Theorem 2.1} and the upper bound on the mixing rate. 
 This concludes the proof.
 \end{proof}

\appendix

\section{Analytic mixing and stationary phase in action-angle variables}

Let $\Omega = \mathbb{S}^1\times (0,\pi/2)$ and $g\in C^1(0,\pi/2)$. We define the weighted Sobolev space $H^1_g$ through the norm
\begin{equation}
	\| f \|_{H^1_g}^2 = \int_\Omega \left(|f(\theta,s)|^2 + |\partial_s f(\theta,s)|^2\right) g(s) \dd s \dd\theta .
\end{equation}
Via a stationary phase argument, the next theorem gives an explicit bound on the solution of a transport equation that is helpful in estimating
the decay of correlation.
\begin{theorem}
	Fix $\delta, \delta'\in (0,1/4)$.
	Let $g, T\in C^1(0,\pi/2)$ be positive functions, and let $f : \RR \times \Omega \to \RR$ solve
	\begin{equation}\label{eq:transportactio}
	\partial_t f + \frac1{T(s)} \partial_\theta f =0\, ,
	\quad 
	\text{for } (t,\theta,I) \in \RR \times \Omega .
	\end{equation}
	Assume further that $\int_{\mathbb{S}^1} f(0,\theta,s)  \dd\theta =0$ for any $s\in (0,1)$. 
	Then, for any  $\varphi \in C^1(\Omega)$ it holds
	\begin{equation}\label{eq:decaycorraction}
		\int_{\mathbb{S}^1\times (\delta, 1-\delta')} f(t,\theta,s)\varphi(\theta,s)g(s)\, \dd s\, \dd\theta
		\le \| f(0,\cdot) \|_{H^1_g} \|\varphi \|_{H^1_g}\,  r(t)
	\end{equation}
    where
    \begin{align}\label{eq:rt}
    	r(t) &= 
    	\frac{1}{t} \left\| g^{-\frac12}\right\|^2_{L^2(\delta,\pi/2-\delta')} 
    	\left(  \frac{T^2g}{|T'|}(\delta) +\frac{T^2g}{|T'|}(\pi/2-\delta') +  \left\| \left( \frac{T^2g}{T'}\right)' \right\|_{L^1(\delta,\pi/2-\delta')} +\left\| \frac{T^2 g^{\frac{1}{2}}}{T'}\right\|_{L^2(\delta,\pi/2-\delta')}\right) \notag
    	\\&\quad+ 
    	\frac{1}{t} \left\| g^{-\frac12}\right\|_{L^2(\delta, \pi/2-\delta')}\left\| \frac{T^2 g^{\frac{1}{2}}}{T'}\right\|_{L^2(\delta,\pi/2-\delta')}\, .
\end{align}
\end{theorem}

\begin{proof}
Expanding in Fourier series in the variable $\theta$ we obtain 
\begin{align}
	f(t, \theta, s)  = \sum_{k\in 4 \mathbb{Z}} f_k(t,s) \e^{i k \theta },\qquad
	f_k(t,s)  = f_k(0,s) \e^{-\frac{ikt}{T(s)}} \, .	
\end{align}
Notice that $f_0(0,s)=0$ for every $s\in (0,1)$.
We compute
\[
\int_\delta^{\pi/2-\delta'} \int_{\mathbb{S}^1}f(t,\theta,s) \varphi(\theta,s) g(s) \dd\theta \dd s =  \sum_k \int_\delta^{\pi/2-\delta'} f_k(t,s) \overline{ \varphi_k(s)} g(s) \dd s
\]
Integrating by parts, for every $k \in 4\mathbb Z$ we have 
\[
\begin{split}
	\int_\delta^{\pi/2-\delta'} f_k(t,s) \overline{ \varphi_k(s)} g(s) \dd s
	= &    \int_\delta^{\pi/2-\delta'} f_k(0,s) \frac1{ikt} \frac{T^2(s)}{T'(s)} \frac{d}{\dd s}\left( \e^{-\frac{ikt}{T(s)}}\right) \overline{ \varphi_k(s)} g(s) \dd s \\
	= &  \left. f_k(0,s) \frac1{ikt} \frac{T^2(s)}{T'(s)} \e^{-\frac{ikt}{T(s)}}\overline{ \varphi_k(s)} g(s) \right|_{s=\delta}^{\pi/2-\delta'} \\
	& -  \frac1{ikt}  \int_\delta^{\pi/2-\delta'}  \partial_s f_k(0,s) \frac{T^2(s)}{T'(s)} \e^{-\frac{ikt}{T(s)}}\overline{ \varphi_k(s)} g(s)\dd s \\
	& -  \frac1{ikt}  \int_\delta^{\pi/2-\delta'}  f_k(0,s) \left(\frac{T^2g}{T'}\right)'(s) \e^{-\frac{ikt}{T(s)}}\overline{ \varphi_k(s)}\dd s \\
	& -   \frac1{ikt} \int_\delta^{\pi/2-\delta'}   f_k(0,s)  \frac{T^2(s)}{T'(s)} \e^{-\frac{ikt}{T(s)}}\overline{ \varphi'_k(s)} g(s)\dd s\, .
\end{split}
\]
We can estimate $\|f_k(0,\cdot)\|_{L^\infty} \le \|f_k(0,\cdot)\|_{L^1}+ \|\partial_sf_k(0,\cdot)\|_{L^1} \le \|f_k(0,\cdot)\|_{H^1_g}\left\| g^{-\frac12}\right\|_{L^2}$ and
$\| \varphi_k\|_{L^\infty}\le\|\varphi_k\|_{H^1_g}\left\| g^{-\frac12}\right\|_{L^2}$
to get
\[
	\left| \int_\delta^{1\pi/2-\delta'} f_k(t,s) \overline{ \varphi_k(s)} g(s) \dd s \right| 
	\le   \frac{r(t)}{|k|}  \|f_k(0,\cdot)\|_{H^1_g} \| \varphi_k\|_{H^1_g} ,
\]
with
\begin{equation}
	r(t) = 
	\frac{1}{t} \left\| g^{-\frac12}\right\|^2_{L^2}
	\left(  \frac{T^2g}{|T'|}(\delta) +\frac{T^2g}{|T'|}(1-\delta') +  \left\| \left( \frac{T^2g}{T'}\right)' \right\|_{L^1} \right)
	 + 
	\frac{1}{t} \left\| g^{-\frac12}\right\|_{L^2}\left\| \frac{T^2 g^{\frac{1}{2}}}{T'}\right\|_{L^2}\, .
\end{equation}
Summing over $k$ we get
\begin{equation}\label{e:mix_stat_phase}
	\begin{split}
	\left|\int_{\mathbb{S}^1\times (\delta, \pi/2-\delta')} f(t,\theta,s)\varphi(\theta,s)g(s)\, \dd s\, \dd\theta \right| 
	\le  
	 \|f(0,\cdot)\|_{H^1_g} \| \varphi\|_{H^1_g} r(t)\, .
	\end{split}
\end{equation}	
This concludes the proof.
\end{proof}

%
%

\subsection{Change of variables}\label{ss:change_of_variables}
The introduction of action-angle variables simplify the structure of a general transport equation to an equation of the form \eqref{eq:transportactio}. 
The estimate \eqref{eq:decaycorraction} then provides a correlation estimate which can be understood in a negative Sobolev space with respect
to the action-angle coordinates. Therefore, an estimate on the change of coordinates may be needed to understand mixing in the usual $H^{-1}$ sense.

\begin{lemma}
Let $\rho =  \rho (t,x)$ and $f(t,\theta,s) = \rho (t,\Phi(\theta,s))$, for some change of coordinate $\Phi$. 
Assume that the Jacobian $J_\Phi(\theta,s)=g(s)$ for some smooth positive function $g$. 
If  
\[
\sup_{\|\phi\|_{H^1_g}\le 1}\int f(t,\theta,s) \phi(\theta,s) g(s) \dd s \le \|f(0,\cdot)\|_{H^1_g}r(t)
\] 
for some $r(t)$, then
\[
\|\rho(t)\|_{H^{-1}}\lesssim \|\rho(0)\|_{H^1} \left(1 + \mathrm{Lip}(\Phi)\right)^2 r(t).
\]
\end{lemma}
\begin{proof}
The proof is a direct computation
\begin{align*}
	\|\rho(t)\|_{H^{-1}} = & \sup_{\|\psi\|_{H^1}\le 1} \int \rho(t,x)\psi(x) \dd x \\
	= &  \sup_{\|\psi\|_{H^1}\le 1} \int f(t,\theta,s) \psi(\Phi(\theta,s)) g(s)  \dd\theta  \dd s \\
	\le & \sup_{\|\phi\|_{H^1_g}\le 1 + \mathrm{Lip}(\Phi) } \int  f(t,\theta,s) \phi(\theta,s) g(s)  \dd\theta ds\\
	\le &  (1 + \mathrm{Lip}(\Phi))  \|f(0,\cdot)\|_{H^1_g}r(t) \\
	\le & (1 + \mathrm{Lip}(\Phi))^2  \|\rho(0,\cdot)\|_{H^1}r(t),
\end{align*}
as needed.
\end{proof}

 \section*{Acknowledgments} 
The research of MCZ was supported by the Royal Society through a University Research Fellowship (URF\textbackslash R1\textbackslash 191492).
EM acknowledges the support received from the European Union's Horizon 2020 research and innovation program under the Marie Sklodowska-Curie grant No. 101025032. EB was supported by the Giorgio and Elena Petronio Fellowship at the Institute for Advanced Study.
\bibliographystyle{abbrv}
\bibliography{biblio_Hamiltonians}
\end{document}